\documentclass[a4paper, 10pt]{amsart}
\usepackage[margin=3.75cm]{geometry}
\usepackage{amsmath,amssymb,amsthm}
\usepackage{esint}
\usepackage{amsaddr}
\usepackage{graphicx}
\usepackage{booktabs}
\usepackage[sort&compress,numbers]{natbib}
\usepackage[colorlinks=true]{hyperref}
\usepackage[capitalize,nameinlink,noabbrev]{cleveref}

\numberwithin{equation}{section}

\usepackage[font=footnotesize,labelfont=bf,width=\textwidth]{caption}

\hypersetup{
    linkcolor = black, 
    filecolor = blue,
    urlcolor  = blue,
    citecolor = blue,
    }
\usepackage[shortlabels]{enumitem}

\crefname{equation}{}{}

\theoremstyle{plain}
\newtheorem{theorem}{Theorem}[section]
\newtheorem{proposition}[theorem]{Proposition}
\newtheorem{lemma}[theorem]{Lemma}

\theoremstyle{definition}

\newtheorem{assumption}{Assumption}[section]

\theoremstyle{remark}
\newtheorem{remark}{Remark}[section]

\AddToHook{env/lemma/begin}{\crefalias{theorem}{lemma}}
\AddToHook{env/proposition/begin}{\crefalias{theorem}{proposition}}
\AddToHook{env/corollary/begin}{\crefalias{theorem}{corollary}}

\makeatletter
\def\@seccntformat#1{%
  \protect\textup{%
    \protect\@secnumfont
    \expandafter\protect\csname format#1\endcsname % <--- added
    \csname the#1\endcsname
    \protect\@secnumpunct
  }%
}

\newcommand{\cal}[1]{\mathcal{#1}}
\newcommand{\R}{\mathbb{R}}
\newcommand{\N}{\mathbb{N}}
\newcommand{\Z}{\mathbb{Z}}
\newcommand{\Pe}{\mathrm{Pe}}
\newcommand{\Hpp}{\cal{H}^{++}}
\newcommand{\Hpm}{\cal{H}^{+-}}
\newcommand{\Hmp}{\cal{H}^{-+}}
\newcommand{\Hmm}{\cal{H}^{--}}
\newcommand{\Hmo}{H^{-1}}
\newcommand{\Hoz}{H_0^1(\Omega)}

\newcommand{\nxi}{n}
\newcommand{\ntheta}{n}
\newcommand{\npsi}{n}

\DeclareMathOperator{\divergence}{div}
\DeclareMathOperator{\linspan}{span}

\newcommand{\QF}{\mathcal{Q}}
\newcommand{\cost}[2]{\avg{f#1 - \frac14 \abs{\nabla #1}^2} - #2}
\renewcommand{\vec}[1]{\mathbf{#1}}
\newcommand{\weakto}{\rightharpoonup}
\newcommand{\abs}[1]{\left\vert #1 \right\vert}
\newcommand{\dx}{\mathrm{d} \vec{x}}
\newcommand{\dr}{\mathrm{d} r}
\renewcommand{\div}{\divergence}
\newcommand{\invlap}{\Delta^{-1}}
\newcommand{\avg}[1]{\int_\Omega #1 \,\dx}
\newcommand{\norm}[1]{\left\| #1 \right\|}
\newcommand{\e}{{\rm e}}
\newcommand{\ocpval}{\mathcal{J}}
\newcommand{\efficiency}{\mathcal{E}}
\newcommand{\lb}{\mathcal{L}}

\newcommand{\Peclet}{P\'eclet}

\title[Optimizing bounds for energy-constrained cooling in 2D]{Optimizing bounds for energy-constrained optimal cooling problems in two dimensions}
\author{Pedro Bl\"oss Braga and Giovanni Fantuzzi}
\address{
    Department of Mathematics
    \\Friedrich-Alexander-Universit\"at Erlangen-N\"urnberg
}
\email{\href{mailto:pedro.bloess@fau.de}{pedro.bloess@fau.de}}
\email{\href{mailto:giovanni.fantuzzi@fau.de}{giovanni.fantuzzi@fau.de}}

\thanks{GF was supported by the DFG-ANR project MONET (grant number 568735456). PB was supported by the DFG project \emph{Renormalized Solutions for the Optimal Control of Evolution Equations} (grant number 551486487).}

\date{\today}
\begin{document}
\begin{abstract}
We study optimal control problems for incompressible fluids in two-dimensional domains with a cold boundary and internal heat sources and sinks. Given a kinetic energy budget, measured by the square of a nondimensional \Peclet\ number $\Pe$, the goal is to maximize a cooling efficiency inversely proportional to the mean square gradient of the fluid's temperature.
Using Lagrange duality, we formulate a well-posed dual problem whose solution yields an upper bound on the maximum cooling efficiency $\efficiency(\Pe)$. We then numerically approximate the dual problem using a convergent hierarchy of semidefinite programs obtained via discretization.
We illustrate this computational approach on optimal cooling problems in a square and in an annulus, explaining also how problem symmetries can be exploited to reduce computational complexity. Finally, we construct admissible points for the dual problem to obtain new analytical upper bounds on the optimal cooling efficiency $\efficiency(\Pe)$. Specifically, we prove that $\efficiency(\Pe) \lesssim \Pe^{2}$ for arbitrary domains and heat distributions, and that $\efficiency(\Pe) \lesssim \Pe^{2}/ \ln^2\Pe$ for cooling flows in disks and annuli with heat source/sink distributions with a uniformly positive azimuthal average. These results generalize and improve known efficiency bounds for energy-constrained cooling flows in a disk.
\end{abstract}

%% Title, with null space to move it up a little
\vspace*{0pt}
\maketitle

%% Main body
\section{Introduction}
Optimal cooling problems, in which the flow of a fluid must be controlled to efficiently cool domains with internal heat sources, 
are fundamental in fluid mechanics. They are central to the design of heat exchangers and also provide a paradigm for studying transport processes.
Locally optimal flows can usually be computed with algorithms for PDE-constrained optimization (see, e.g., \cite{Hassanzadeh2014,Alben2017a,Alben2017b,Motoki2018,Souza2020}), and provide lower bounds on the optimal cooling efficiency. Upper bounds, in contrast, are substantially more difficult to compute because they are statements about the efficiency of all possible flows.
Remarkably, the difficulty persists even in a highly simplified setting where the admissible fluid flows are steady, heat is a passive scalar with no effect on the fluid's motion, and the only constraints on the flow are incompressibility and a bound on some norm of the fluid's velocity. In this work, we focus on two-dimensional incompressible flows with prescribed $L^2$ norm, which physically correspond to flows with a prescribed kinetic energy.

\subsection{Problem description}
Consider an incompressible fluid flowing inside a bounded, possibly multiply connected domain $\Omega \subset \R^2$ with Lipschitz boundary $\partial\Omega$. The flow is described by a time-independent and weakly divergence-free velocity field $\vec{u} \in L^2(\Omega; \R^2)$ satisfying the no-penetration condition $\smash{\vec{u}\cdot \hat{\vec{n}}=0}$ on $\partial\Omega$, where $\smash{\hat{\vec{n}}}$ is the outward-pointing unit vector normal to the boundary and $\smash{\vec{u}\cdot \hat{\vec{n}}}$ is the normal trace of $\vec{u}$.
The domain is subject to a nonzero distribution $f \in H^{-1}(\Omega)$ of heat sources and sinks, while its boundary is held at a constant temperature of zero. The temperature $T$ of the fluid then solves the steady advection-diffusion equation
\begin{equation}\label{e:ade}
    \begin{cases}
        - \Delta T + \vec{u}\cdot \nabla T =  f &\text{in }\Omega,\\
        T=0 &\text{on }\partial\Omega.
    \end{cases}
\end{equation}

We wish to understand how well the domain can be cooled by flows with bounded kinetic energy, measured by the square of a nondimensional parameter $\Pe$ called the \Peclet\ number. We consider a flow efficient at cooling if the temperature is nearly constant on $\Omega$, so $\smash{\avg{\abs{\nabla T}^2}}$ is small. Optimal flows then solve
\begin{equation}\label{e:ocp}\tag{OCP}
    \ocpval(\Pe):=
    \inf_{\substack{
        \vec{u}\in L^2(\Omega;\R^2) \\ 
        \div\vec{u}=0 \text{ in }\Omega\\ 
        \vec{u} \cdot \hat{\vec{n}} = 0  \text{ on }\partial\Omega\\
        \norm{\vec{u}}_{L^2} \leq \,\Pe
    }}\;
    \avg{\abs{\nabla T}^2}
\end{equation}
This is a nonconvex optimal control problem for the advection-diffusion equation \cref{e:ade}, where the nonconvexity is due entirely to the bilinear term $\vec{u}\cdot\nabla T$ in this equation. We associate to the optimal control cost $\ocpval(\Pe)$ the optimal cooling efficiency
\begin{equation}\label{e:efficiency}
    \efficiency(\Pe) := \frac{\|f\|_{\Hmo}^2}{ \ocpval(\Pe)},
\end{equation}
where the factor $\|f\|_{\Hmo}^2$ is included to have $\efficiency(0) = 1$. In other words, we measure the optimal cooling efficiency relative to the no-flow state. 

The definitions immediately imply that $\efficiency(\Pe) \geq 1$. Better lower bounds can be found by constructing particular flows, including locally optimal ones computed with PDE-constrained optimization algorithms. Our aim is to obtain complementary upper bounds by deriving positive lower bounds on $\ocpval(\Pe)$.

\subsection{Known bounds}
When $\Omega$ is a disk and the heating is uniform ($f=1$), numerical optimization produces convection rolls achieving $\smash{\avg{\abs{\nabla T}^2}\sim \Pe^{-1}}$~\cite{Marcotte2018}.\footnote{Throughout the paper, given nonnegative quantities $f$ and $g$, we write $f\lesssim g$ if $f\leq C g$ for some positive constant $C$, $g \gtrsim f$ if $f \lesssim g$, and $f \sim g$ if both $f\lesssim g$ and $g \lesssim f$.} Inspired by these computations, analytical constructions for a broad class of sufficiently smooth heating distributions $f$ provide the lower bound $\efficiency(\Pe) \gtrsim \Pe$ at sufficiently large \Peclet\ numbers, with a prefactor depending only on the heating distribution~\cite[Proposition~5.1]{Tobasco2022}.
This estimate is in fact true for a time-dependent generalization of our optimal cooling problem, which we do not consider for simplicity. 

Lower bounds on optimal cooling efficiencies have also been derived for variations of problem \cref{e:ocp}, such as optimal cooling in a disk under the `enstrophy' constraint $\smash{\norm{\nabla \vec{u}}_{L^2}\leq\Pe}$ \cite{Tobasco2022}; optimal cooling of uniformly heated flows between an insulated boundary and an isothermal boundary, with constraints on the $L^p$ and $W^{1,p}$ norms of $\vec{u}$ \cite{Iyer2022}; optimal cooling in insulated domains with balanced heat sources and sinks \cite{Shaw2007,Song2023}; and optimal cooling between differentially heated plates \cite{Hassanzadeh2014,TobascoDoering2017,DoeringTobasco2019,Kumar2024}. In all cases, lower bounds on the cooling efficiency are derived from particular flows that resemble those obtained with numerical optimization \cite{Motoki2018,Souza2020,Marcotte2018}.

Finding matching upper bounds on $\efficiency(\Pe)$ has proven difficult. The problem has received attention within a broader study of heat transport and mixing of passive scalars (see, e.g., \cite{Doering1996,Hassanzadeh2014,Motoki2018,TobascoDoering2017,DoeringTobasco2019,Souza2020,Kumar2024,Chanillo2025,Song2023,Shaw2007,Thiffeault2004,DoeringThiffeault2006,Thiffeault2012}) and upper bounds on cooling efficiencies have been derived for many of the settings described above \cite{Tobasco2022,Shaw2007,Song2023,Kumar2024}. However, they often exceed the corresponding lower bounds. In particular, the upper bound $\efficiency(\Pe) \lesssim \Pe^2$ for the energy-constrained problem \cref{e:ocp} in a disk~\cite[Proposition~2.6]{Tobasco2022} does not match the lower bound $\efficiency(\Pe) \gtrsim \Pe$. It is then natural to wonder if the proofs of these upper bounds can be improved and, in particular, whether numerical computations can be used to guide such improvements.

\subsection{Contributions}
We combine Lagrange duality with semidefinite programming to optimize lower bounds on the optimal control cost $\ocpval(\Pe)$, which translate into upper bounds on the cooling efficiency $\efficiency(\Pe)$ via \cref{e:efficiency}. Our key insight is to recognize that existing proofs of lower bounds for $\ocpval(\Pe)$ implicitly construct feasible points for the Lagrange dual of problem \cref{e:ocp}, which is an infinite-dimensional semidefinite program (SDP). We explicitly derive this dual problem, prove its well-posedness (\cref{prop:ldp-well-posed}), and show that it can be approximated with arbitrary accuracy by finite-dimensional SDPs (\cref{thm:convergence}). These are computationally tractable for fixed $\Pe$, so the best bound provable via Lagrange duality can be computed numerically.

We also use the dual of \cref{e:ocp} to derive new analytical upper bounds on $\efficiency(\Pe)$. First, we give an elementary proof that
\begin{equation}\label{e:an-bound-intro}
    \efficiency(\Pe) \leq 1 + C \Pe^2
    \qquad \forall\, \Pe \geq 0,
\end{equation}
where $C>0$ depends only on the fluid's domain $\Omega$ (\cref{th:analytical-lb}). This result applies to arbitrary heat distributions $f$ and domains $\Omega$, generalizing a similar bound proved in \cite{Tobasco2022} when $\Omega$ is a disk. 
We then fix $\Omega$ to be a disk or an annulus and prove that
\begin{equation}\label{e:intro-log-lb}
    \efficiency(\Pe) \leq C \, \frac{\Pe^2}{\ln^2 \Pe}
    \qquad \forall\, \Pe > 1
\end{equation}
if $f \in L^1(\Omega)\cap H^{-1}(\Omega)$ has a uniformly positive average on any circle around the domain's center, where the constant $C$ depends on the domain and on $f$ (\cref{th:refined-annulus,th:refined-disk}). This logarithmically improved bound holds in particular for uniform heating ($f=1$). Its proof is strongly inspired by the numerical results produced by our semidefinite programming computations for a uniformly heated annulus.

\subsection{Relation to prior work}
Lagrange duality is a classical approach for computing bounds on optimization problems. While it is often left implicit, it underpins most existing bounds on the cooling efficiency or other heat transport properties of incompressible flows. In particular, duality arguments similar to ours have been applied to optimal cooling problems in insulated domains with balanced heating and cooling \cite{Song2023}. What distinguishes our work is that we explicitly formulate the Lagrange dual of problem \cref{e:ocp} in a form suitable for a numerical implementation based on SDPs.

The connection between optimal cooling problems and SDPs is not surprising: discretizations of problem \cref{e:ocp} are nonconvex quadratic programs, and the Lagrange dual of a quadratic program is an SDP~\cite{Shor1987,Nesterov2000}. This fact has been exploited to compute performance bounds for quadratic-program discretizations of optimal design problems in structural mechanics and photonics~\cite{Tyburec2019,Angeris2019,Angeris2021,Gertler2025,Dalklint2025}. Our approach is similar, but we dualize before discretizing. Exchanging these operations has a twofold advantage: we can prove that SDP solutions converge to the best bound provable with Lagrange duality, and we can derive rigorous bounds analytically. 

Finally, our approach resembles SDP implementations of the \emph{background method}, a variational method for bounding the time-averaged dissipation of incompressible flows (see \cite{Doering1994,Constantin1995,Doering1996} or the review \cite{Fantuzzi2022} and the references therein). This is because the variational problem at the heart of the background method is the Lagrange dual of a quadratic program relaxation of the optimization problem for the time-averaged dissipation over flows satisfying the incompressible Navier--Stokes equations~\cite{Kerswell1998,Kerswell1999}.

\subsection{Outline}
The plan of the paper is as follows. In \cref{s:dual-problem}, we derive and analyze the Lagrange dual of problem \cref{e:ocp}. 
In particular, we construct an admissible point for the dual problem to obtain an analytical lower bound on $\ocpval(\Pe)$, which, combined with \cref{e:efficiency}, leads to the efficiency bound in \cref{e:an-bound-intro}. In \cref{s:discretization}, we introduce SDP discretizations of the Lagrange dual problem and prove that they converge with increasing resolution for all discretization schemes with strong approximation properties in the Sobolev space $H^1(\Omega)$. This convergence is not unexpected, but the proof requires specialized arguments due to the non-standard, `infinite-dimensional SDP' structure of the Lagrange dual problem. 
We demonstrate our numerical strategy in \cref{s:numerics} on energy-constrained cooling problems in a square and in an annulus, outlining also how problem symmetries can be exploited to reduce computational complexity. 
In \cref{s:improved-bounds}, inspired by our computations, we refine the analysis of optimal cooling in disks and annuli, proving the efficiency bound in \cref{e:intro-log-lb}. \Cref{s:conclusion} offers final remarks and perspectives for future work.
\section{The Lagrange dual problem}\label{s:dual-problem}

In this section, we formulate and study the Lagrange dual of the optimal cooling problem \cref{e:ocp}. To obtain this dual problem in a convenient form for both analysis and computations, in \cref{ss:streamfunction} we reformulate \cref{e:ocp} in terms of a scalar streamfunction. We derive the dual of the reformulated problem in \cref{ss:dualization}, study its optimizers in \cref{ss:well-posedness}, and use it to deduce the efficiency bound \cref{e:an-bound-intro} in \cref{ss:an-lower-bound}.

\subsection{Streamfunction formulation}\label{ss:streamfunction}
Let $\Gamma_0$ denote the exterior boundary of the domain $\Omega$ and, if this is multiply connected, let $\Gamma_1,\ldots,\Gamma_p$ denote all other connected components of the boundary (see \cref{f:domain-with-holes} for an illustration). Let $\mathcal{H}$ be the subspace of $H^1(\Omega)$ containing functions that are constant on each piece of the boundary, i.e.,
\begin{equation*}
    \mathcal{H} := \left\{ 
    \psi \in H^1(\Omega):
    \;
    \exists 
    c_0,\ldots,c_p \in \R 
    \text{ s.t. } 
    \psi|_{\Gamma_i} = c_i ,\; 0\leq i \leq p 
    \right\}.
\end{equation*}
Every weakly divergence-free vector field $\vec{u}\in L^2(\Omega;\R^2)$ satisfying the no-penetration boundary condition $\vec{u}\cdot \hat{\vec{n}}=0$ can be expressed as
\begin{equation}\label{e:streamfunction-def}
    \mathbf{u} = \Pe\, \nabla^\perp \psi = \Pe\left( \partial_y \psi,\, -\partial_x \psi\right)
\end{equation}
for a streamfunction $\psi \in \mathcal{H}$ determined up to an additive constant (see, e.g., \cite[\S3.1]{GiraultRaviart1979}). We fix this constant such that $\psi|_{\Gamma_0} = 0$, so $\mathcal{H}=H^1_0(\Omega)$ if $\Omega$ is simply connected. The factor of $\Pe$ in \cref{e:streamfunction-def} is included for convenience.

\begin{figure}[t]
    \centering
    \includegraphics[scale=0.85]{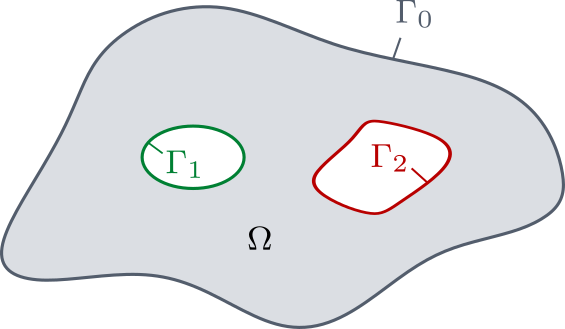}
    \caption{Illustration of a multiply connected domain $\Omega$ (shaded gray region) with exterior boundary $\Gamma_0$ (thick gray line) and interior boundaries $\Gamma_1$ (thick green line) and $\Gamma_2$ (thick red line).}
    \label{f:domain-with-holes}
\end{figure}

With this re-parametrization, the optimal cooling problem \cref{e:ocp} becomes
\begin{equation}\label{e:ocp-psi}
    \ocpval(\Pe)=
    \inf_{\substack{
        \psi \in \mathcal{H} \\ 
        \norm{\nabla \psi}_{L^2} \leq 1
    }}\;
    \avg{\abs{\nabla T}^2},
\end{equation}
where $T$ is the unique weak solution of the advection-diffusion equation \cref{e:ade} with velocity field $\vec{u}=\Pe \nabla^\perp \psi$. Specifically, $T$ is the unique function in $H^1_0(\Omega)$ satisfying
\begin{equation}\label{e:ade-weak}
    \avg{ \nabla \xi \cdot \nabla T + \Pe\,\xi \,\nabla^\perp \psi \cdot \nabla T - f\xi}= 0 \quad \forall \xi \in H^1_0(\Omega).
\end{equation}
The existence and uniqueness of such a weak solution follow from standard arguments based on the Lax--Milgram theorem; the only difficulty is to show that the integral of $\xi \,\nabla^\perp \psi \cdot \nabla T$ is well defined despite the low regularity of the factors, but this is true because $\nabla^\perp \psi \cdot \nabla T = \partial_y\psi \partial_x T - \partial_x \psi \partial_y T$ is a Jacobian determinant. In particular, as discussed in  \cite[Remark~2.2]{Tobasco2022} and \cite[\S3.1]{Song2023}, analysis in \cite{Coifman1990} implies that
\begin{equation}\label{e:clms}
    \left\vert 
    \avg{ \xi \,\nabla^\perp \psi \cdot \nabla T } 
    \right\vert
    \leq c(\Omega) \norm{\nabla \xi}_{L^2} \norm{\nabla \psi}_{L^2} \norm{\nabla T}_{L^2}
\end{equation}
for a positive constant $c(\Omega)$ depending only on the domain $\Omega$.
This estimate will play a key role also in our numerical analysis of \cref{ss:convergence}.

\subsection{The dual problem}\label{ss:dualization}
We now dualize the minimization problem in \cref{e:ocp-psi} to obtain a convex maximization problem whose optimal value, denoted by $\lb(\Pe)$, is a lower bound on $\ocpval(\Pe)$.  To ease the notation, we associate to every pair $(a,\xi) \in \R \times H^1_0(\Omega)$ a homogeneous quadratic form $\mathcal{Q}(a,\xi):\mathcal{H} \times H^1_0(\Omega) \to \R$ defined as
\begin{equation}\label{e:qf}
    \mathcal{Q}(a,\xi)[\psi,\theta] := \avg{ a|\nabla \psi|^2 + \abs{\nabla\theta}^2 - \Pe\,\xi\, \nabla^\perp \psi \cdot \nabla \theta}.
\end{equation}
Note that the map $(a,\xi) \mapsto \mathcal{Q}(a,\xi)$ is affine and that the integral is well defined thanks to estimate \cref{e:clms}. We write $\mathcal{Q}(a,\xi)\succeq 0$ to indicate that $\mathcal{Q}(a,\xi)$ is positive semidefinite, meaning that $\mathcal{Q}(a,\xi)[\psi,\theta] \geq 0$ for all $\psi \in \mathcal{H}$ and $\theta \in H^1_0(\Omega)$.

\begin{proposition}\label{th:duality}
    For every $\Pe\geq 0$, the optimal value $\ocpval(\Pe)$ of problem \cref{e:ocp} satisfies $\ocpval(\Pe) \geq \lb(\Pe) \geq 0$, where
    \begin{equation}\label{e:ldp}\tag{LDP}
        \lb(\Pe) := \max_{\substack{
            a\in \R \\ 
            \xi \in H^1_0(\Omega) \\ 
            \QF(a,\xi) \succeq  0 
        }}\;
        \avg{f\xi - \frac14 \abs{\nabla \xi}^2} - a.
    \end{equation}
\end{proposition}

\begin{remark}
    For $\Pe > 0$, we expect the inequality $\ocpval(\Pe) \geq \lb(\Pe)$ to be strict because the optimal cooling problem \cref{e:ocp} is not convex. For $\Pe=0$, \cref{e:ocp} and \cref{e:ldp} are trivially solvable and attain the same value (see \cref{rem:Pe=0} below).
\end{remark}

\begin{remark}\label{rem:symmetrization-1}
    The constraint $\QF(a,\xi) \succeq  0$ requires $a\geq 0$, and for $\Pe>0$ we can in fact restrict the optimization in \cref{e:ldp} to $a>0$. We can then set $a=\alpha^2$ with $\alpha > 0$ and perform the change of variables $\psi \mapsto \psi / \sqrt{\alpha}$, $\theta \mapsto \theta \sqrt{\alpha}$ to conclude that $\QF(a,\xi)$ is positive semidefinite if and only if so is the `balanced' quadratic form
    \begin{equation}\label{e:qf-symmetrized}
    \mathcal{Q}'(\alpha,\xi)[\psi,\theta] := \avg{ \alpha|\nabla \psi|^2 + \alpha\abs{\nabla\theta}^2 - \Pe\,\xi\, \nabla^\perp \psi \cdot \nabla \theta}.
    \end{equation}
    Consequently,
    \begin{equation*}
        \lb(\Pe) = \max_{\substack{
            \alpha \in \R \\ 
            \xi \in H^1_0(\Omega) \\ 
            \QF'(\alpha,\xi) \succeq  0 
        }}\;
        \avg{f\xi - \frac14 \abs{\nabla \xi}^2} - \alpha^2.
    \end{equation*}
    We will use this reformulation in \cref{s:improved-bounds} to prove \cref{e:intro-log-lb}. We also remark that the feasible set of this balanced problem formulation is a cone in $\R\times\Hoz$, but exploring the computational implications of this fact is beyond the scope of our present work.
\end{remark}

\begin{remark}\label{rem:symmetrization-2}
    For simply connected $\Omega$, $\mathcal{H}=\Hoz$ and the quadratic form $\QF'(\alpha,\xi)$ satisfies $\mathcal{Q}'(\alpha,\xi)[\psi,\theta] = \mathcal{Q}'(\alpha,\xi)[\theta,-\psi]$. As demonstrated in \cref{ss:example-square}, this identity can afford computational savings if combined with other problem symmetries.
\end{remark}

\begin{proof}[Proof of \cref{th:duality}]
    The temperature $T$ in \cref{e:ocp} is the unique function in $H^1_0(\Omega)$ satisfying \cref{e:ade-weak}. Viewing the test function $\xi$ in that equation as a Lagrange multiplier for the advection-diffusion equation \cref{e:ade}, and introducing a Lagrange multiplier $a\geq 0$ for the constraint $\norm{\nabla\psi}_{L^2}\leq 1$, we can write
    \begin{align*}
        \ocpval(\Pe)
        &= 
        \inf_{\substack{
        \psi \in \mathcal{H} \\ 
        T \in H^1_0(\Omega)
        }}\;
        \sup_{\substack{a \in \R \\ \xi \in H^1_0(\Omega) \\ a\geq 0}}
        \avg{\abs{\nabla T}^2 + a |\nabla \psi|^2 - (\nabla\xi + \Pe\,\xi\,\nabla^\perp\psi )\cdot \nabla T + f\xi} - a
        \\
        &\geq 
        \sup_{\substack{a \in \R \\ \xi \in H^1_0(\Omega) \\ a\geq 0}}
        \inf_{\substack{
        \psi \in \mathcal{H} \\ 
        T \in H^1_0(\Omega)
        }}\;
        \avg{\abs{\nabla T}^2 + a |\nabla \psi|^2 - (\nabla\xi + \Pe\,\xi\,\nabla^\perp\psi ) \cdot \nabla T + f\xi} - a.
    \end{align*}
    The inequality on the second line is due to switching minimization and maximization.

    Next, we perform the inner minimization over $T$ for fixed $\psi$, $a$ and $\xi$. The optimal temperature $T^*$ solves the Euler--Lagrange equation
    $-2\Delta T^* + \Delta \xi + \Pe \, \nabla \xi \cdot \nabla^\perp \psi = 0$
    with boundary condition $T^* = 0$. Writing $\invlap$ for the inverse Laplacian with homogeneous Dirichlet boundary conditions, we obtain
    \begin{equation}\label{e:Topt}
        T^* =  \frac12\xi + \frac{\Pe}{2}\ \invlap(\nabla \xi \cdot \nabla^\perp \psi).
    \end{equation}
    To compute the value of the minimum, we multiply the Euler--Lagrange equation by $T^*$ and integrate by parts over $\Omega$ to find that
    \begin{equation}\label{e:el-1}
    \avg{ (\nabla\xi + \Pe\,\xi\,\nabla^\perp\psi )\cdot \nabla T^*} = 2 \avg{\abs{\nabla T^*}^2}.
    \end{equation}
    We then calculate
    \begin{equation}\label{e:el-2}
        \avg{\nabla \xi \cdot \nabla \invlap(\nabla \xi \cdot \nabla^\perp \psi)}
        = -\avg{ \xi \nabla \xi \cdot \nabla^\perp \psi}
        = 0,
    \end{equation}
    using integration by parts, the divergence-free nature of $\nabla^\perp \psi$, and the zero boundary values of $\xi$. With \cref{e:Topt}, \cref{e:el-1} and \cref{e:el-2}, we obtain
    \begin{align*}
        \ocpval(\Pe)
        \geq 
        \sup_{\substack{a \in \R \\ \xi \in H^1_0(\Omega) \\ a\geq 0}}
        \inf_{\substack{
        \psi \in \mathcal{H}
        }}\;
        \avg{
            a |\nabla \psi|^2 
            - \frac14 \abs{\nabla \xi}^2 
            - \frac{\Pe^2}{4} \abs{\nabla \invlap(\nabla \xi \cdot \nabla^\perp \psi)}^2 
            + f\xi
        } 
        - a.
    \end{align*}

    We now use the identity 
    \begin{equation*}
        -\avg{ \frac{\Pe^2}{4} \abs{\nabla \invlap(\nabla \xi \cdot \nabla^\perp \psi)}^2 } = \min_{\theta \in H^1_0(\Omega)} \avg{ \abs{\nabla\theta}^2 - \Pe\,\xi\,\nabla^\perp \psi \cdot \nabla\theta }
    \end{equation*}
    to remove the term involving the inverse Laplacian and arrive at the lower bound
    \begin{align*}
        \ocpval(\Pe)
        \geq 
        \sup_{\substack{a \in \R \\ \xi \in H^1_0(\Omega) \\ a\geq 0}}
        \inf_{\substack{
        \psi \in \mathcal{H} \\ 
        \theta \in H^1_0(\Omega)
        }}
        \avg{a\abs{\nabla\psi}^2 + \abs{\nabla\theta}^2 - \Pe\,\xi\,\nabla^\perp\psi \cdot \nabla\theta - \frac{1}{4}\abs{\nabla \xi}^2 + f\xi} - a.
    \end{align*}
    The minimax problem on the right-hand side of this inequality is equivalent to the maximization problem \cref{e:ldp} because the inner minimum is zero if the quadratic form $\QF(a,\xi)$ defined in \cref{e:qf} is positive semidefinite, and is negative infinity otherwise. Therefore, the outer maximization selects pairs $(a,\xi)$ satisfying $\QF(a,\xi)\succeq 0$. Since this constraint automatically implies that $a\geq 0$, the latter condition need not be enforced in \cref{e:ldp} and we conclude that $\ocpval(\Pe) \geq \lb(\Pe)$. Finally, $\lb(\Pe)\geq 0$ because the choices $a=0$ and $\xi=0$ are feasible for \cref{e:ldp} and have a zero objective value.
\end{proof}

\subsection{Optimizers and strictly feasible near-optimizers}\label{ss:well-posedness}
In stating problem \cref{e:ldp}, we tacitly assumed that the maximum is attained. We prove this fact next, establishing also that the maximizer is unique.

\begin{proposition}\label{prop:ldp-well-posed}
    Problem \cref{e:ldp} has a unique optimal solution.
\end{proposition}

\begin{proof}
    Let $\smash{\{(a_n,\xi_n)\}_{n\in\N}}$ be a maximizing sequence for \cref{e:ldp}. The semidefiniteness of $\smash{\QF(a_n,\xi_n)}$ implies that $a_n\geq 0$, and since $\lb(\Pe)\geq 0$ we may assume without loss of generality that $\smash{\cost{\xi_n}{a_n}} \geq 0$. Then, 
    \begin{align}\label{e:boundedness}
        a_n + \frac14\norm{\nabla \xi_n}_{L^2}^2 
        \leq \avg{f\xi_n}
        \leq \norm{f}_{\Hmo} \norm{\nabla \xi_n}_{L^2} 
        \leq 2\norm{f}_{\Hmo}^2 + \frac18\norm{\nabla \xi_n}_{L^2}^2
    \end{align}
    for every $n$. This inequality implies that $\{(a_n,\xi_n)\}_{n\in\N}$ is uniformly bounded, so we may replace it with a subsequence converging to a limit $(a^*,\xi^*)$ weakly in $\R \times H^1_0(\Omega)$. 
    
    We claim this limit is an optimizer for \cref{e:ldp}. To verify its feasibility, fix arbitrary $\psi\in\mathcal{H}$ and $\theta \in H^1_0(\Omega)$. The functional $(a,\xi)\mapsto \QF(a,\xi)[\psi,\theta]$ is affine by \cref{e:qf} and bounded by \cref{e:clms}, so weak convergence and the semidefiniteness of $\QF(a_n,\xi_n)$ give
    \begin{equation*}
        \QF(a^*,\xi^*)[\psi,\theta] = \lim_{n\to\infty}  
        \QF(a_n,\xi_n)[\psi,\theta]
        \geq 0.
    \end{equation*}
    To verify optimality, instead, note that the function $(a,\xi)\mapsto \smash{\cost{\xi}{a}}$ is weakly upper semicontinuous in $\R \times H^1_0(\Omega)$. Then,
    \begin{equation*}
        \lb(\Pe) = \lim_{n\to\infty} \cost{\xi_n}{a_n} \leq \cost{\xi^*}{a^*} \leq \lb(\Pe),
    \end{equation*}
    from which we conclude that $(a^*,\xi^*)$ is optimal for \cref{e:ldp}.

    Finally, we show that $(a^*,\xi^*)$ is the only optimal solution. The strict concavity of the cost function with respect to $\xi$ implies that if $(a',\xi')$ is another optimizer for \cref{e:ldp}, then $\xi'=\xi^*$. But then, we must also have $a'=a^*$ because the optimality of these constants for \cref{e:ldp} implies that
    \begin{align*}
        a' 
        = 
        \sup_{\psi,\theta}
        \frac{\avg{\Pe\,\xi'\,\nabla^\perp\psi\cdot\nabla\theta - \abs{\nabla\theta}^2}}{\avg{\abs{\nabla \psi}^2}}
        =
        \sup_{\psi,\theta}
        \frac{\avg{\Pe\,\xi^*\,\nabla^\perp\psi\cdot\nabla\theta - \abs{\nabla\theta}^2}}{\avg{\abs{\nabla \psi}^2}}
        = a^*,
    \end{align*}
    where the suprema are taken over $\psi \in \mathcal{H}$ and $\theta \in \Hoz$. 
\end{proof}

Finally, with a view towards the numerical analysis of \cref{e:ldp} in \cref{ss:convergence}, we construct near-optimal admissible points for which $\QF(a,\xi)$ is positive definite.

\begin{proposition}\label{prop:near-optimizers}
    Fix $\Pe\geq 0$. For every $\varepsilon \in (0,1]$, there exist $a \in \R$ and $\xi \in H^1_0(\Omega)$ such that $\smash{\cost{\xi}{a} \geq \lb(\Pe) - \varepsilon}$ and such that
    \begin{gather*}
        \QF(a, \xi)[\psi,\theta] \geq  \frac{\varepsilon}{1+\lb(\Pe)}\avg{\abs{\nabla\psi}^2 + \abs{\nabla\theta}^2}
        \quad\forall(\psi,\theta) \in \mathcal{H} \times \Hoz.
    \end{gather*}
\end{proposition}
\begin{proof}
    Let $(a^*,\xi^*)$ solve \cref{e:ldp}. For $\delta \in [0,1]$, the function $\xi = (1-\delta) \xi^*$ and the constant $a = (1-\delta) a^* + \delta$ satisfy
    $\smash{\cost{\xi}{a} \geq \lb(\Pe) - \delta \left(1 + \lb(\Pe) \right)}$
    and
    \begin{align*}
        \QF(a,\xi)[\psi,\theta] 
        &= (1-\delta)\QF(a^*,\xi^*)[\psi,\theta] 
        + \delta \avg{\abs{\nabla\psi}^2 + \abs{\nabla\theta}^2}.
    \end{align*}
    Since $\QF(a^*,\xi^*)$ is positive semidefinite, the claim follows upon taking $\smash{\delta = \frac{\varepsilon}{1+\lb(\Pe)}}$, which is an admissible choice because $\lb(\Pe)\geq 0$.
\end{proof}

\subsection{An explicit optimal cooling efficiency bound}\label{ss:an-lower-bound}
We conclude this section by using the dual problem \cref{e:ldp} to prove the optimal cooling efficiency bound in \cref{e:an-bound-intro}. The proof relies on the identities 
\begin{equation*}
    \norm{f}_{\Hmo} 
    := \sup_{\substack{\varphi \in \Hoz\\\norm{\nabla \varphi}_{L^2}=1}} \abs{\avg{f \varphi}}
    = \left( \avg{ \abs{ \nabla \invlap f }^2 } \right)^\frac12.
\end{equation*}

\begin{theorem}\label{th:analytical-lb}
    Let $c=c(\Omega)$ be the constant in \cref{e:clms}. For every $\Pe\geq 0$,
    \begin{equation}\label{e:an-bound-theorem}
        \ocpval(\Pe)\geq \lb(\Pe)\geq \frac{\|f\|_{\Hmo}^2}{ 1 + c^2 \Pe^2}.
    \end{equation}
    Consequently, $\efficiency(\Pe) \leq 1 + c^2 \Pe^2$.
\end{theorem}

\begin{remark}
    This result extends to general domains and arbitrary $\Pe$ the bound obtained in \cite[Proposition~2.6]{Tobasco2022} when $\Omega$ is a disk and $\Pe$ is sufficiently large. The analysis in \cite{Tobasco2022}, however, applies also to time-dependent optimal cooling problems. We expect a version of \cref{th:analytical-lb} to hold for time-dependent problems, too.
\end{remark}
\begin{remark}\label{rem:Pe=0}
    The inequalities in \cref{th:analytical-lb} are equalities if $\Pe=0$, because in this case problem \cref{e:ocp} is solved by $\vec{u}=0$, giving $T=-\invlap f$ and $\ocpval(0) = \|f\|_{\Hmo}^2$.
\end{remark}

\begin{proof}
    The upper bound on $\efficiency(\Pe)$ follows from \cref{e:efficiency} and \cref{e:an-bound-theorem}.
    The first inequality in \cref{e:an-bound-theorem} is ensured by \cref{th:duality}. To prove the second one, let us fix
    \begin{equation*}
        \xi = -\frac{2 \invlap f}{1 + c^2 \Pe^2}
        \quad\text{and}\quad
        a = \frac{c^2 \Pe^2}{(1 + c^2 \Pe^2)^2}\,\norm{f}_{\Hmo}^2.
    \end{equation*}
    The identities 
    \begin{equation}\label{e:xi2f}
        c \,\Pe \norm{\nabla \xi}_{L^2} = \frac{2 c \,\Pe}{1 + c^2 \Pe^2} \norm{ f }_{\Hmo} = 2\sqrt{a}
    \end{equation}
    imply that
    $a \norm{\nabla \psi}_{L^2}^2 + \norm{\nabla \theta}_{L^2}^2 - c\, \Pe \norm{ \nabla \xi }_{L^2} \norm{\nabla \psi}_{L^2} \norm{\nabla \theta}_{L^2} \geq 0$ 
    because the expression on the left-hand side is a perfect square. This inequality ensures that the quadratic form $\QF(a,\xi)$ in \cref{e:qf} is positive semidefinite by estimate \cref{e:clms}, so the chosen $a$ and $\xi$ are admissible for problem \cref{e:ldp}.
    We then conclude that
    \begin{equation*}
        \lb(\Pe) 
        \geq \avg{ f\xi - \frac14 \abs{\nabla\xi}^2} - a
        = \frac{\|f\|_{\Hmo}^2}{ 1 + c^2 \Pe^2},
    \end{equation*}
   where the last step uses the definitions of $\xi$ and $a$, the first identity in \cref{e:xi2f}, and the identity $\smash{\avg{f \invlap f} = - \norm{f}_{\Hmo}^2}$.
\end{proof}
\section{SDP discretizations}\label{s:discretization}

We now show that problem \cref{e:ldp} can be approximated numerically by a hierarchy of SDPs obtained via discretization.

\subsection{A general discretization scheme}
To approximate problem \cref{e:ldp} numerically, we discretize the space $\Hoz$ in which the function $\xi$ is optimized as well as the space $\mathcal{H}\times \Hoz$ on which the quadratic form $\QF(a,\xi)$ must be positive semidefinite. We do so using a generic conforming Galerkin discretization scheme, which in practice could be a pseudo-spectral or a finite element method.
Specifically, for each $n \in \N$, we introduce $n$-dimensional subspaces\footnote{We take these subspaces to have the same dimension for notational simplicity, but there are no obstacles to considering subspaces with different dimensions. In this case, the convergence results stated in \cref{ss:convergence} apply with $n$ equal to the smallest subspace dimension.}
\begin{align*}
    \Xi_n &\subset \Hoz, &
    \Theta_n &\subset \Hoz, &
    \Psi_n &\subset \mathcal{H}.
\end{align*}
Then, we define the discrete version of \cref{e:ldp} to be
\begin{equation}\label{e:ldp-discrete}
    \lb_n(\Pe) := \max_{\substack{a\in\R \\ \xi_n \in \Xi_n \\ \QF_n(a,\xi_n)\succeq 0}} \
    \avg{f\xi_n - \frac14 \abs{\nabla \xi_n}^2} - a,
\end{equation}
where $\QF_n$ is the restriction of the quadratic form $\QF$ to $\Psi_n \times \Theta_n$ and the constraint $\QF_n(a,\xi_n)\succeq 0$ simply means that $\QF(a,\xi_n)[\psi_n, \theta_n]\geq 0$ for all $\psi_n \in \Psi_n$ and $\theta_n \in \Theta_n$. The existence of a unique maximizer for \cref{e:ldp-discrete} follows exactly as in \cref{prop:ldp-well-posed}, and $\lb_n(\Pe)\geq 0$ because the choices $a=0$ and $\xi_n = 0$ are admissible for \cref{e:ldp-discrete}.

We stress that $\lb_n(\Pe)$ only approximates the optimal value $\lb(\Pe)$ of \cref{e:ldp}, but is neither an upper bound nor a lower bound. This is because we restrict the feasible set of \cref{e:ldp} by optimizing $\xi$ only over the subspace $\Xi_n$, but we also relax the constraint $\QF(a,\xi)\succeq 0$ by imposing it only over $\Psi_n \times \Theta_n$. Lower bounds on $\lb(\Pe)$ could be computed using null Lagrangian translations or dual `occupation measure' relaxations~\cite{Chernyavsky2023,Valmorbida2016,Korda2022}, but we do not presently know if these lower bounds converge. In contrast, as we prove in \cref{ss:convergence}, the approximations given by \cref{e:ldp-discrete} converge under mild and reasonable assumptions on the subspaces $\Xi_n$, $\Theta_n$ and $\Psi_n$.

\subsection{Explicit SDP formulation}
The discrete problem \cref{e:ldp-discrete} is an SDP with a quadratic objective function. To make this evident, we introduce a basis $\{u_1,\ldots,u_n\}$ for $\Xi_n$, a basis $\{v_1,\ldots,v_n\}$ for $\Theta_n$, a basis $\{w_1,\ldots,w_n\}$ for $\Psi_n$, and expand
\begin{align*}
        \xi_n(\vec{x}) &= \sum_{i=1}^{\nxi} z_i u_i(\vec{x}), &
        \theta_n(\vec{x}) &= \sum_{i=1}^{\ntheta} t_i v_i(\vec{x}), &
        \psi_n(\vec{x}) &= \sum_{i=1}^{\npsi} p_i w_i(\vec{x}).
\end{align*}
We collect the expansion coefficients in $n$-dimensional column vectors $\vec{z}$, $\vec{t}$ and $\vec{p}$, define the $n$-dimensional column vector $\vec{f}=(\smash{\avg{f u_1}},\ldots,\smash{\avg{f u_n}})$, and consider the $n \times n$ matrices $A$, $B$, $C$ and $D_1,\ldots,D_n$ with entries
\begin{align*}
    A_{jk} &= \avg{\nabla u_j \cdot \nabla u_k}, &
    B_{jk} &= \avg{\nabla v_j \cdot \nabla v_k}, \\
    C_{jk} &= \avg{\nabla w_j \cdot \nabla w_k}, &
    (D_i)_{jk} &= \avg{u_i \nabla v_k \cdot \nabla^\perp w_j}.
\end{align*}
Using these quantities, we can write
${\avg{f\xi_n} = \vec{f}^\top \vec{z} }$, ${\avg{ \abs{\nabla \xi_n}^2 } = \vec{z}^\top A \vec{z}}$, and
\begin{equation*}
    \QF(a,\xi_n)[\psi_n, \theta_n] 
    =
    \begin{bmatrix}
        \vec{p} \\ \vec{t}
    \end{bmatrix}^\top
    \begin{bmatrix}
        a C &\ -\frac\Pe2 \sum_{i} z_i D_i \\
        -\frac\Pe2 \sum_{i} z_i D_i^\top & B
    \end{bmatrix}
    \begin{bmatrix}
        \vec{p} \\ \vec{t}
    \end{bmatrix}.
\end{equation*}
This quantity is nonnegative for all $\vec{p}$ and $\vec{t}$ (equivalently, for all $\psi_n\in\Psi_n$ and $\theta_n \in \Theta_n$) if and only if the $2\times 2$ block matrix on the right-hand side is positive semidefinite. Therefore, the `abstract' discrete problem \cref{e:ldp-discrete} is equivalent to
\begin{align}\label{e:sdp}
    \max_{\substack{a \in \R \\ \vec{z}\in\R^n}} \quad
    &\vec{f}^\top \vec{z} - \frac14 \vec{z}^\top A \vec{z} - a
    \quad\text{s.t.}\quad
     \begin{bmatrix}
        a C &\ -\frac\Pe2 \sum_{i} z_i D_i \\
        -\frac\Pe2 \sum_{i} z_i D_i^\top & B
    \end{bmatrix}
    \succeq 0.
\end{align}
This is a semidefinite program with a concave quadratic objective and can be solved by software packages such as \texttt{Clarabel}~\cite{Clarabel2024}. Linear conic program solvers such as \texttt{Mosek}~\cite{mosek} can be used after lifting the term $\vec{z}^\top A \vec{z}$ into a second-order cone constraint.

\subsection{Convergence analysis}\label{ss:convergence}
We now prove that the discrete dual problem \cref{e:ldp-discrete} converges to \cref{e:ldp} as the dimension of the subspaces $\Xi_n$, $\Theta_n$ and $\Psi_n$ increases. The only requirement, stated precisely in \cref{ass:density} below, is the existence of `recovery sequences' with elements in these subspaces that approximate functions in $\Hoz$ and $\mathcal{H}$. This property is typical of pseudo-spectral and finite element discretizations.

\begin{assumption}\label{ass:density}
The following conditions hold:
\begin{enumerate}[a), noitemsep, widest={a)}, leftmargin=*]
    \item $\forall \xi \in \Hoz$, $\exists \{\xi_n\}_{n\in\N}$ with $\xi_n \in \Xi_n$ such that $\lim_{n\to\infty}\|\xi_n - \xi\|_{\Hoz} = 0$.
    \item $\forall \theta \in \Hoz$, $\exists \{\theta_n\}_{n\in\N}$ with $\theta_n \in \Theta_n$ such that $\lim_{n\to \infty}\|\theta_n - \theta\|_{\Hoz} = 0$.
    \item $\forall  \psi \in \mathcal{H}$, $\exists \{\psi_n \}_{n \in \N}$ with $\psi_n \in \Psi_n$ such that $\lim_{n\to\infty}\|\psi_n - \psi\|_{H^1(\Omega)} = 0$.
\end{enumerate}
\end{assumption}

To simplify the presentation, we also assume that the integrals required to set up problem \cref{e:ldp-discrete} are computed exactly. This assumption is mild, and quadrature errors that decrease sufficiently fast with $n$ could be accounted for if desired. With these assumptions in place, we are ready to state and prove the main result of this section.

\begin{theorem}\label{thm:convergence}
   Fix $\Pe\geq 0$. Let $\lb(\Pe)$ and $(a^*,\xi^*)$ be the optimal value and optimizer of problem \cref{e:ldp}. For each $n\in \N$, let $\lb_n(\Pe)$ and $(a^*_n,\xi^*_n)$ be the optimal value and optimizer of problem \cref{e:ldp-discrete}. If \cref{ass:density} holds, then 
   \begin{enumerate}[{\rm a)}, noitemsep, widest={a)}, leftmargin=*]
       \item\label{item:L-conv}  $\lb_n(\Pe) \to \lb(\Pe)$ as $n\to\infty$,
       \item\label{item:a-conv}  $a_n^* \to a^*$ as $n\to\infty$,
       \item\label{item:xi-conv}  $\xi_n^* \to \xi^*$ strongly in $\Hoz$ as $n\to\infty$.
   \end{enumerate}
\end{theorem}

\begin{proof}
The estimates in \cref{e:boundedness} ensure that the sequences $\{a^*_n\}$ and $\{\xi^*_n\}$ are uniformly bounded in $\R$ and $\Hoz$, so there exist subsequences $\smash{\{a^*_{n_i}\}}$ and $\smash{\{\xi^*_{n_i}\}}$ such that $\smash{a^*_{n_i}\to a'}$ and $\smash{\xi^*_{n_i}} \weakto \xi'$ weakly in $\Hoz$ for some $a'$ and $\xi'$. We claim, and will prove later, that
\begin{subequations}
    \begin{gather}
        \label{e:claim1}
        \lb(\Pe) \leq \liminf_{n\to\infty} \lb_n(\Pe) \\
        \label{e:claim2}
        \QF(a',\xi')\succeq 0.
    \end{gather}
\end{subequations}
The second claim means that the pair $(a',\xi')$ is feasible for \cref{e:ldp}. Then, since the function $(a,\xi)\mapsto \smash{\cost{\xi}{a}}$ is weakly upper semicontinuous, we can start from \cref{e:claim1} to estimate
\begin{align}\label{e:liminf-limsup-combo}
    \lb(\Pe)
    \leq \liminf_{n\to\infty} \lb_n(\Pe)
    &\leq \liminf_{n_i\to\infty} \cost{\xi_{n_i}^*}{a_{n_i}^*}
    \\ \nonumber
    &\leq \limsup_{n_i\to\infty} \cost{\xi_{n_i}^*}{a_{n_i}^*}
    \\ \nonumber
    &
    \leq \cost{\xi'}{a'}
    \leq \lb(\Pe).\phantom{\int_\Omega}
\end{align}
These inequalities must be equalities, so the weak subsequence limits $a'$ and $\xi'$ are optimizers of \cref{e:ldp}. Since this problem has unique optimizers $a^*$ and $\xi^*$, and since the arguments above apply to every convergent subsequence of discrete optimizers, we conclude that $a_n^* \to a^*$ and $\xi^*_n \weakto \xi^*$ weakly in $\Hoz$. We also conclude that the inequalities in \cref{e:liminf-limsup-combo} hold without passing to subsequences, so $\lb_n(\Pe)\to\lb(\Pe)$. 

To complete the proof, we now need to prove the claimed inequalities \cref{e:claim1} and \cref{e:claim2}, as well as the strong convergence $\xi^*_n \to \xi^*$ in $\Hoz$. 

To prove strong convergence, we observe that the weak convergence $\smash{\xi^*_n \weakto \xi^*}$ and the limit $\smash{\lb_n(\Pe)\to\lb(\Pe)}$ imply that $\smash{\avg{\abs{\nabla\xi^*_n}^2}\to \avg{\abs{\nabla\xi^*}^2}}$. Then, since $\smash{\Hoz}$ is a uniformly convex Banach space, $\{\xi_n^*\}$ converges strongly \cite[Prop.~3.32]{Brezis2011}.

Next, we prove inequality \cref{e:claim1}. Fix $\varepsilon \in (0,1)$ and let $(a_\varepsilon,\xi_\varepsilon)$ be a strictly feasible near-optimizer for \cref{e:ldp} constructed using \cref{prop:near-optimizers}. By \cref{ass:density}, there exist $\xi_n \in \Xi_n$ such that $\xi_n \to \xi_\varepsilon$ strongly in $\Hoz$. We claim that the pair $(a_\varepsilon, \xi_n)$ is feasible for the discrete dual problem \cref{e:ldp-discrete} for all sufficiently large $n$. Indeed, for every pair $(\psi,\theta) \in \Psi_n \times \Theta_n \subset \mathcal{H}\times \Hoz$ we can use \cref{e:clms} to estimate
\begin{align*}
    \QF(a_\varepsilon, \xi_n)[\psi,\theta] 
    &= \QF(a_\varepsilon, \xi_\varepsilon)[\psi,\theta] + \Pe \avg{(\xi_\varepsilon-\xi_n) \nabla^\perp \psi \cdot \nabla\theta}
    \\
    &\geq \frac{\varepsilon}{1+\mathcal{L(\Pe)}} \avg{\abs{\nabla\psi}^2 + \abs{\nabla\theta}^2 \!} - 
    c \Pe \norm{\nabla \xi_\varepsilon - \nabla \xi_n}_{L^2} \norm{\nabla \psi}_{L^2} \norm{\nabla\theta}_{L^2}
    \\
    &\geq \left(\frac{\varepsilon}{1+\mathcal{L(\Pe)}} - \frac{c\Pe}{2} \norm{\nabla \xi_\varepsilon - \nabla \xi_n}_{L^2} \right)\avg{\abs{\nabla\psi}^2 + \abs{\nabla\theta}^2}
\end{align*}
and the last quantity is positive for large $n$ since $\xi_n \to \xi_\varepsilon$ in $\Hoz$. We then estimate
\begin{align*}
    \liminf_{n \to \infty}\lb_n(\Pe)
    &\geq \liminf_{n \to \infty} \cost{\xi_n}{a_\varepsilon} 
    \\
    &
    = \cost{\xi_\varepsilon}{a_\varepsilon} 
    \\
    &
    \geq \lb(\Pe) - \varepsilon
\end{align*}
and let $\varepsilon \to 0$ to obtain \cref{e:claim1}.

Finally, we establish \cref{e:claim2}. Fix $\psi \in \mathcal{H}$ and $\theta \in \Hoz$. We need to show that $\QF(a', \xi')[\psi, \theta] \geq 0$. To do this, recall that $\xi'$ and $a'$ are the weak limit of subsequences of $\{\xi_n^*\}$ and $\{a_n^*\}$, which we also index by $n$ below to lighten the notation. \Cref{ass:density} provides `recovery sequences' $\{\psi_n\}$ and $\{\theta_n\}$ with $\psi_n\in \Psi_n$ and $\theta_n \in \Theta_n$ that converge strongly to $\psi$ and~$\theta$. Then, for each $n$, we estimate
\begin{align}\label{e:Q-estimate}
    \Big| \QF(a', \xi')[\psi, \theta] - \QF(a_n^*, \xi_n^*)[\psi_n, \theta_n] \Big|
    \leq
    &\abs{  \avg{a' |\nabla \psi|^2 
        -  a_n^* |\nabla \psi_n|^2} }
    \\ \nonumber
    + &\abs{ \avg{ \abs{\nabla \theta}^2 - \abs{\nabla \theta_n}^2}}
    \\ \nonumber
    + &\abs{ \avg{ \xi' \,\nabla^{\perp} \psi \cdot \nabla \theta
            - \xi^*_n \,\nabla^{\perp} \psi_n \cdot \nabla \theta_n  } } \Pe.
\end{align}
The first two terms on the right-hand side tend to zero with increasing $n$ because $a_n^* \to a'$, $\psi_n \to \psi$ strongly, and $\theta_n \to \theta$ strongly. The last term also tends to zero, but showing this requires more care because the convergence $\xi_n^* \weakto \xi'$ is only weak. (We cannot rely on strong convergence because we derived strong convergence using \cref{e:claim2}, which is what we are proving.) To handle this difficulty, we first estimate
\begin{align*}
    \abs{ 
        \avg{ \xi' \nabla^{\perp} \psi \cdot \nabla \theta
        - \xi^*_n \nabla^{\perp} \psi_n \cdot \nabla \theta_n
        }
    }
    \leq 
    &\abs{
        \avg{
            (\xi' - \xi_n^*) \nabla^{\perp} \psi \cdot \nabla \theta
        }
    }
    \\
    +
    &\abs{
        \avg{
            \xi^*_n \nabla^{\perp} (\psi - \psi_n) \cdot \nabla \theta
        }
    }
    \\
    + 
    &\abs{
        \avg{
            \xi^*_n \nabla^{\perp}\psi_n \cdot \nabla (\theta - \theta_n)
        }
    }
    .
\end{align*}
We then apply estimate \cref{e:clms} to the last two terms in this inequality and obtain
\begin{align*}
    \abs{ 
        \avg{ \xi' \nabla^{\perp} \psi \cdot \nabla \theta
        - \xi^*_n \nabla^{\perp} \psi_n \cdot \nabla \theta_n
        }
    }
    \lesssim
    &\abs{
        \avg{
            (\xi' - \xi_n^*) \nabla^{\perp} \psi \cdot \nabla \theta
        }
    }
    \\
    + 
    &\| \nabla \xi_n^* \|_{L^2}
    \| \nabla \psi - \nabla\psi_n \|_{L^2} 
    \| \nabla \theta \|_{L^2}  
    \\
    + 
    &\| \nabla \xi_n^* \|_{L^2}
    \| \nabla \psi_n \|_{L^2} 
    \| \nabla \theta - \nabla \theta_n \|_{L^2}
    .
\end{align*}
The weak convergence $\xi_n^* \weakto \xi'$ in $\smash{\Hoz}$, the strong convergence $\psi_n\to \psi$ in $\smash{H^1(\Omega)}$, and the strong convergence $\theta_n \to \theta$ in $\smash{\Hoz}$ ensure that all three terms on the right-hand side of this inequality tend to zero. Thus, the last term on the right-hand side of \cref{e:Q-estimate} vanishes with increasing $n$. The first two terms in that inequality also vanish, so
$\QF(a', \xi')[\psi, \theta] = 
    \lim_{n \to \infty}  \; \QF(a_n^*, \xi_n^*)[\psi_n, \theta_n]
    \geq 0$
because $a_n^*$ and $\xi_n^*$ are feasible for \cref{e:ldp-discrete}. We have thus proved \cref{e:claim2}, completing the proof of \cref{thm:convergence}.
\end{proof}
\section{Numerical demonstrations}\label{s:numerics}

We now demonstrate the computational approach of \cref{s:discretization} by optimizing bounds on the optimal cost of problem \cref{e:ocp} for two heat distributions in the square $\Omega=\smash{(-1,1)^2}$ and for a uniformly heated annulus. 
In both cases, we exploit symmetries of the domain and of the heat distribution to replace the dual problem \cref{e:ldp} with equivalent problems, whose SDP discretizations have a lower computational complexity (see \cite{Fantuzzi2022,Farid2026} for details of symmetry exploitation in similar problems). Our convergence analysis extends to these symmetry-reduced formulations with straightforward modifications. We implemented all SDPs in Julia using \texttt{JuMP}~\cite{jump} and \texttt{Mosek}~\cite{mosek}.

\subsection{Optimal cooling in a square}\label{ss:example-square}
For the square domain $\Omega = (-1,1)^2$, we write $\vec{x}=(x,y)$ and consider two heat source distributions, $f(\vec{x})=1$ and $f(\vec{x}) = 1 + \cos(\frac\pi2 x) \cos(\frac{3\pi}{2} y)$. With the convention of \cref{ss:streamfunction}, the streamfunction space is $\mathcal{H} = \Hoz$ because the domain is simply connected.

\subsubsection{Symmetry reduction}
\label{ss:example-square-symm}
Since the domain and the heat distributions we consider are invariant under the transformations $(x, y)\mapsto(-x, y)$ and $(x, y)\mapsto(x, -y)$, we can simplify problem \cref{e:ldp} before discretizing it into an SDP. Specifically, we decompose $H^1_0(\Omega)$ into the direct sum of the subspaces
\begin{gather*}
    \Hpp = \left\{u \in H^1_0(\Omega): u(-x, y)=u(x, -y)=u(x, y)\right\},\\
    \Hpm = \left\{u \in H^1_0(\Omega): u(-x, y)=u(x, y),\, u(x, -y)=-u(x, y)\right\},\\
    \Hmp = \left\{u \in H^1_0(\Omega): u(-x, y)=-u(x, y),\, u(x, -y)=u(x, y)\right\},\\
    \Hmm = \left\{u \in H^1_0(\Omega): u(-x, y)=u(x, -y)=-u(x, y)\right\},
\end{gather*}
and we rewrite problem \cref{e:ldp} as
\begin{align}\label{e:ldp-square-symm}
    \lb(\Pe) = \max_{\substack{a \in \R \\ \xi \in \Hpp}}\quad
    &\int_\Omega f\xi - \frac14 \abs{\nabla \xi}^2 \,\dx - a
    \\
    \nonumber
    \text{s.t.}\quad
    &\QF(a,\xi)[\psi,\theta] \geq 0 \quad \forall (\psi,\theta) \in \Hpp \times \Hmm,\\\nonumber
    &\QF(a,\xi)[\psi,\theta] \geq 0 \quad \forall (\psi,\theta) \in \Hpm \times \Hmp,\\\nonumber
    &\QF(a,\xi)[\psi,\theta] \geq 0 \quad \forall (\psi,\theta) \in \Hmm \times \Hpp,\\\nonumber
    &\QF(a,\xi)[\psi,\theta] \geq 0 \quad \forall (\psi,\theta) \in \Hmp \times \Hpm.
\end{align}
Indeed, the restriction of the optimization to $\xi \in \Hpp$ is warranted by a standard symmetrization argument and the convexity of \cref{e:ldp}. Given this restriction, it suffices to check that $\QF(a,\xi) \succeq 0$ on the listed subspace combinations because the symmetries of $\psi$ and $\theta$ for other combinations imply $\smash{\avg{\xi\, \nabla^\perp \psi \cdot \nabla \theta}}=0$, so $\QF(a,\xi)[\psi,\theta]\geq0$.

The last two constraints in \cref{e:ldp-square-symm} are actually redundant. To see why, note that the positive lower bound on $\lb(\Pe)$ from \cref{th:analytical-lb} implies that, for $\Pe>0$, the optimal $a$ is strictly positive. Then, given $\psi \in \Hmm$ and $\theta \in \Hpp$, we have that $\QF(a,\xi)[\psi,\theta]=\QF(a,\xi)[\psi',\theta']$ for $\psi'=\theta / \sqrt{a}\in \Hpp$ and $\theta'=-\sqrt{a}\psi \in \Hmm$, so the first and third constraints are equivalent. The same is true for the second and fourth constraints. In our implementation, therefore, we consider only the first two constraints. This is an example of the computational savings mentioned in \cref{rem:symmetrization-2}.

\subsubsection{Discretization scheme}
We discretize problem \cref{e:ldp-square-symm} using a spectral method whose basis functions are the sinusoidal eigenfunctions of the Dirichlet Laplacian on $\Omega$. Specifically, given a positive integer $n\in \N$, we set $\omega = \frac\pi2$ and replace the spaces $\Hpp$, $\Hpm$, $\Hmp$ and $\Hmm$ with their $n^2$-dimensional subspaces
\begin{gather*}
    \Hpp_n = \linspan\left\{
            \cos [(2k-1) \omega x]
            \cos [(2\ell-1) \omega y]
            :\; k,\ell \in \{ 1,\ldots, n \}
        \right\},
    \\
    \Hpm_n = \linspan\left\{
            \cos[(2k-1) \omega x]
            \sin\left(2\ell \omega y\right)
            :\; k,\ell \in \{ 1,\ldots, n \}
        \right\},
    \\
    \Hmp_n = \linspan\left\{
            \sin\left(2k \omega x\right)
            \cos[(2\ell-1) \omega y]
            :\;k,\ell \in \{ 1,\ldots, n \}
        \right\},
    \\
    \Hmm_n = \linspan\left\{
            \sin\left(2k \omega x\right)
            \sin\left(2\ell \omega y\right)
            :\; k,\ell \in \{ 1,\ldots, n \}
        \right\}.
\end{gather*}
With these restrictions, problem \cref{e:ldp-square-symm} (without the last two redundant constraints) becomes an SDP with $n^2+1$ variables and two LMIs of size $2n^2 \times 2n^2$.

\begin{table}[t]
    \centering
    \caption{
        \label{tab:square-results}
        Lower bound approximations $\lb_n(\Pe)$ for optimal cooling in a square, computed at $\Pe=1$, $10$ and $20$ for two heat distributions $f$ and increasing values of the discretization parameter $n$.
    }
    \begin{tabular}{c c ccc c ccc}
        \toprule
        && \multicolumn{3}{c}{$f(\vec{x}) = 1$}
        && \multicolumn{3}{c}{$f(\vec{x}) = 1 + \cos(\frac\pi2 x) \cos(\frac{3\pi}{2} y)$}
        \\
        \cmidrule{3-5} \cmidrule{7-9}
        $n$
        && $\lb_n(1)$  & $\lb_n(10)$  & $\lb_n(20)$ 
        && $\lb_n(1)$  & $\lb_n(10)$  & $\lb_n(20)$ 
        \\[0.75ex]
        1 && 0.530145 & 0.365785 & 0.188599  && 0.530145 & 0.365785 & 0.188599 \\ 
        2 && 0.554924 & 0.424603 & 0.270803  && 0.550771 & 0.388327 & 0.232529 \\ 
        3 && 0.558572 & 0.441014 & 0.300804  && 0.554392 & 0.404774 & 0.264730 \\ 
        4 && 0.559561 & 0.443606 & 0.314503  && 0.555358 & 0.405228 & 0.273432 \\ 
        5 && 0.559932 & 0.444797 & 0.319255  && 0.555728 & 0.406098 & 0.277742 \\ 
        10 && 0.560286 & 0.445523 & 0.322775  && 0.556082 & 0.406505 & 0.279781 \\ 
        15 && 0.560323 & 0.445563 & 0.322890  && 0.556119 & 0.406543 & 0.279859 \\ 
        20 && 0.560333 & 0.445572 & 0.322903  && 0.556128 & 0.406552 & 0.279870 \\ 
        25 && 0.560336 & 0.445576 & 0.322907  && 0.556131 & 0.406555 & 0.279873 \\ 
        30 && 0.560337 & 0.445577 & 0.322908  && 0.556133 & 0.406557 & 0.279875 \\
        \bottomrule
    \end{tabular}
\end{table}

\subsubsection{Results}
For each of the two heat distributions reported at the start of this section, we solved SDP discretizations of problem \cref{e:ldp-square-symm} at $\Pe=1$, $10$ and $20$ with discretization parameter $n$ ranging from $1$ to $30$. The corresponding lower bound approximations $\lb_n(\Pe)$ are reported in \cref{tab:square-results} and converge quickly as $n$ is increased.  One expects fast convergence from a spectral discretization even though the analysis in \cref{s:discretization} does not provide any convergence rates. It is also clear from the table that higher values of $\Pe$ necessitate a higher resolution for a fixed level of accuracy.

\begin{figure}
    \centering
    \includegraphics[scale=0.8]{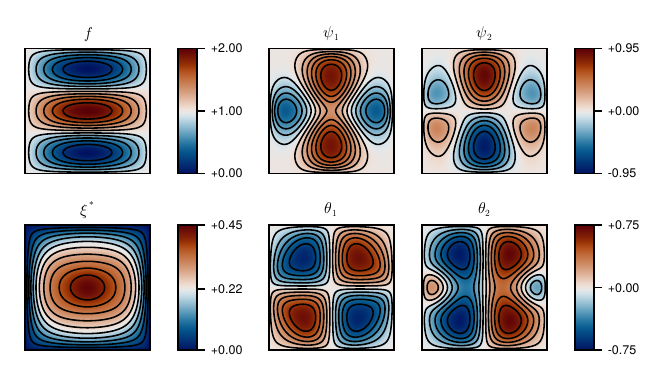}
    \vskip-10pt
    \caption{
    \label{fig:square-sin-critical-fields-Pe10}
    Results for optimal cooling in a square with a sinusoidal heat distribution, obtained at $\Pe=10$ with discretization parameter $n=30$.
    \emph{Top-left:} The heat distribution $f$.
    \emph{Bottom-left:} The optimal field $\xi^*$ for \cref{e:ldp-square-symm}.
    \emph{Top-right and bottom-right:} Critical functions $(\psi_i, \theta_i)$, for which the first two constraints $\QF(a^*,\xi^*)[\psi_i,\theta_i]\geq 0$ in \cref{e:ldp-square-symm} are active.
    }

    \centering
    \includegraphics[scale=0.8]{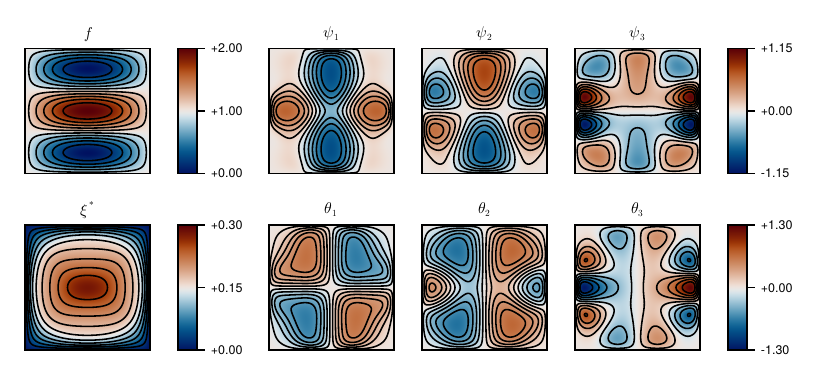}
    \vskip-10pt
    \caption{
    \label{fig:square-sin-critical-fields-Pe20}
    Results for optimal cooling in a square with a sinusoidal heat distribution, obtained at $\Pe=20$ with discretization parameter $n=30$.
    Conventions are as in \cref{fig:square-sin-critical-fields-Pe10}.
    }
\end{figure}

The need for finer discretization at larger $\Pe$ becomes clear upon inspecting the optimal solution of problem \cref{e:ldp-square-symm}. For brevity, we consider only the sinusoidal heat distribution at $\Pe=10$ and $\Pe=20$; results for uniform heating are similar. \Cref{fig:square-sin-critical-fields-Pe10,fig:square-sin-critical-fields-Pe20} show that the optimizer $\xi^*$ has a simple structure at both $\Pe$ values. The same is not true of the `critical functions' $(\psi_i,\theta_i)$ for which the first two constraints in \cref{e:ldp-square-symm} are active, i.e., for which $\QF(a^*,\xi^*)[\psi_i,\theta_i]=0$. (Critical functions for the two redundant constraints are equivalent up to the symmetry transformation described at the end of \cref{ss:example-square-symm}, so we do not discuss them.) These critical functions can be computed by minimizing $\QF(a^*,\xi^*)$ over $\Hpp \times \Hmm$ and $\Hpm \times \Hmp$ subject to the unit-norm constraint $\smash{\avg{\psi^2 + \theta^2}=1}$, which amounts to finding the principal eigenfunctions of the linear operator associated with $\QF(a^*,\xi^*)$. As \cref{fig:square-sin-critical-fields-Pe10,fig:square-sin-critical-fields-Pe20} show, new critical functions with finer spatial structure emerge as $\Pe$ increases, and they must be resolved accurately to obtain a converged solution.
Note that the emergence of new critical functions as $\Pe$ increases resembles what happens in the background method for bounding time averages in turbulent flows \cite{Doering1996,Plasting2003,fpw2018,Wen2015}.

\subsection{Optimal cooling in a uniformly heated annulus}\label{ss:annulus}
Next, we consider optimal cooling in the annulus $\Omega = \{\vec{x}\in\R^2:\;\rho < \abs{\vec{x}} < 1\}$ with $\rho=\frac14$ and a uniform heat distribution $f(\vec{x})=1$. It is convenient to work in the polar coordinates $(r,\phi)$.

\subsubsection{Symmetry reduction}\label{ss:annulus-symmetrization}
Since the chosen heat distribution depends only on the radial direction, we may restrict the optimization to functions $\xi=\xi(r)$ varying only in the radial direction. Then, a Fourier series expansion in polar coordinates detailed in \cref{app:fourier-annulus} allows us to rewrite problem \cref{e:ldp} as
\begin{equation}\label{e:ldp-annulus-symm}
    \lb(\Pe) = 
    \max_{\substack{a \in \R \\ \xi \in H^1_0(\rho,1) \\ \QF_k(a,\xi)\succeq 0 \,\forall k\in\Z_{\neq0}}}\quad
    2\pi\int_\rho^1  \left( f\xi - \frac{1}{4} \abs{\xi'}^2 \right) r\dr - a,
\end{equation}
where $k$ is the wavenumber of each Fourier mode in the expansion, $\Z_{\neq0}=\Z\setminus\{0\}$, and $\QF_k(a,\xi)$ is a quadratic form on $H^1_0(\rho,1) \times H^1_0(\rho,1)$ defined as
\begin{equation}\label{e:qf-fourier-annulus}
    \QF_k(a,\xi)[\psi,\theta] = \int_\rho^1 
        \left[ a \left(r|\psi'|^2 + \frac{k^2}{r}|\psi|^2\right) +
         r|\theta'|^2 + \frac{k^2}{r} |\theta|^2
         + k\Pe\,\xi'\,\psi \theta  \right] \dr.
\end{equation}

Problem~\cref{e:ldp-annulus-symm} can be simplified further by setting $a=\alpha^2$, $\psi=(u+v)/\sqrt{\alpha}$ and $\theta = (u-v)\sqrt{\alpha}$, which is a version of the change of variables discussed in \cref{rem:symmetrization-1}. Indeed, a direct calculation yields $\QF_k(a,\xi)[\psi,\theta] = 2\mathcal{S}_k^+(\alpha,\xi)[u] + 2\mathcal{S}_k^-(\alpha,\xi)[v]$ where
\begin{align*}
    \mathcal{S}_k^\pm(\alpha,\xi)[u] = \int_\rho^1 
        \alpha \left(r|u'|^2+ \frac{k^2}{r}|u|^2 \right)
        \pm \frac{\Pe}2 \,k\, \xi' u^2  \, \dr.
\end{align*}
Since $\smash{\mathcal{S}_k^-=\mathcal{S}_{-k}^+}$, and since $u$ and $v$ can be chosen independently, $\QF_k(a,\xi) \succeq 0$ on $H^1_0(\rho,1) \times H^1_0(\rho,1)$ for every $k\in\Z_{\neq0}$ if and only if $\mathcal{S}_k^+(\alpha,\xi) \succeq 0$ on $H^1_0(\rho,1)$ for every $k\in\Z_{\neq0}$. Consequently, the maximization problem in \cref{e:ldp-annulus-symm} is equivalent to
\begin{align}\label{e:ldp-annulus-symm-v2}
    \max_{\substack{\alpha \in \R \\ \xi \in H^1_0(\rho,1) \\ \mathcal{S}_k^+(\alpha,\xi) \succeq 0 \,\forall k \in \Z_{\neq 0}}}\quad
    &2\pi\int_\rho^1  \left( f\xi - \frac{1}{4} \abs{\xi'}^2 \right) r\dr - \alpha^2.
\end{align}

\subsubsection{Discretization scheme}
To discretize problem \cref{e:ldp-annulus-symm-v2}, we first fix $m\in\N$ and impose the constraint $\mathcal{S}_k^+(\alpha,\xi) \succeq 0$ only for nonzero $k \in \{-m,\ldots,m\}$. We then discretize the radial direction using a pseudo-spectral Chebyshev collocation method with $n+2$ grid points, that is, we replace $H^1_0(\rho,1)$ with the space of polynomials of degree $n+1$ defined by their values at $n+2$ Chebyshev points of the second kind. We compute the integrals in \cref{e:ldp-annulus-symm-v2} and in the definition of $\mathcal{S}_k^+(\alpha,\xi)$ using a Clenshaw--Curtis quadrature at $3(n+2)$ Chebyshev points of the second kind. This quadrature is exact for the polynomials $r f\xi$, $r|\xi'|^2$, $r|u'|^2$, and $\xi'u^2$, and is also accurate for the rational function $u^2/r$. 

Upon accounting for the boundary conditions satisfied by functions in $H^1_0(\rho,1)$, this discretization strategy results in an SDP with $n+1$ optimization variables and $2m$ LMIs of size $n\times n$ each. We consider $m$ and $n$ sufficiently large if halving them changes the numerical approximation of $\lb(\Pe)$ by less than $0.1\%$. With this criterion, the largest \Peclet\ number we could consider on a workstation with $500$ GB of RAM before exceeding the available memory was $\Pe = 354.8$, which required $(m,n)=(256,64)$.

To access higher \Peclet\ numbers, we used an alternative approach where we impose the constraint $\mathcal{S}_k^+(\alpha,\xi) \succeq 0$ for all $k\in\R$. After discretization in the radial direction, this stronger constraint becomes an $n\times n$ quadratic polynomial matrix inequality with indeterminate $k$, and can be enforced computationally using techniques for sum-of-squares programming (see, e.g., \cite[\S3.3.2]{Parrilo2013}). The resulting optimization problem is an SDP with one positive semidefinite cone of dimension $2n\times 2n$, which can be implemented using the Julia package \texttt{SumOfSquares.jl} \cite{SumOfSquares1,SumOfSquares2}. Again, we consider $n$ sufficiently large if halving it changes the SDP's optimal value by less than $0.1\%$. This strategy allowed us to access \Peclet\ numbers up to $\Pe=1000$, which required $n=256$ and (at this resolution) approximately 15 hours per SDP.

\begin{figure}[t]
    \centering
    \includegraphics[width=\linewidth]{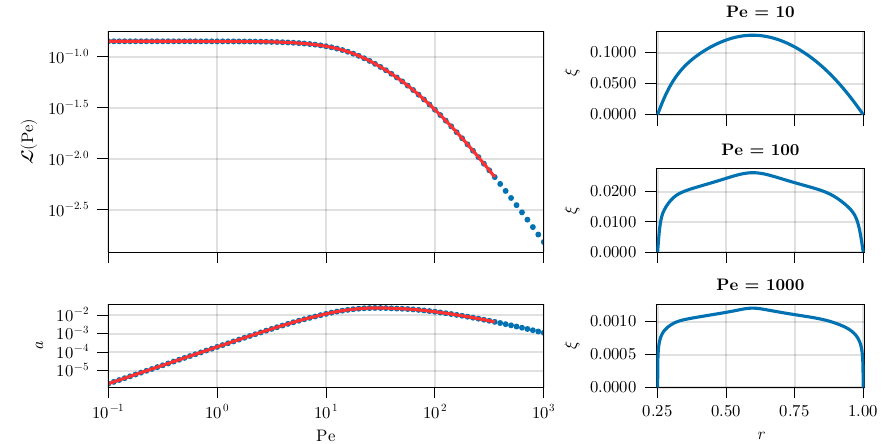}
    \caption{Optimal solutions of SDP discretizations of problem \cref{e:ldp-annulus-symm-v2}, which approximate lower bounds on the optimal value $\ocpval(\Pe)$ of problem \cref{e:ocp} for a uniformly heated annulus. Results are computed by imposing the constraint $\mathcal{S}_k^+(\alpha,\xi) \succeq 0$ for $k\in\{-m,\ldots,m\}$ (red lines) and for all $k\in\R$ (blue dots). \emph{Top-left:} Approximations of $\lb(\Pe)$ vs. $\Pe$. \emph{Bottom-left:} Optimal $a=\alpha^2$ vs. $\Pe$. \emph{Right column:} Optimal $\xi$ for $\Pe=10$, $100$ and $1000$. Note the different ranges of the vertical axis.
    \label{f:bounds_vs_Pe_annulus}
    }
\end{figure}

\subsubsection{Results}
We used the two discretization schemes described above to approximate $\lb(\Pe)$ for \Peclet\ numbers ranging from $\Pe=0.1$ to $\Pe=1000$. The results are plotted in the top-left panel of \cref{f:bounds_vs_Pe_annulus}. Other panels in that figure show the optimal constant $a$ as a function of $\Pe$ and the optimal function $\xi(r)$ for $\Pe=10$, $100$ and $1000$. Note that the results obtained with the discrete wavenumbers $k\in\Z$ are indistinguishable from those obtained with $k\in\R$.

Our numerical approximations to $\lb(\Pe)$ are consistent with the functional form of the analytical lower bound from \cref{th:analytical-lb} in the sense that they approach a constant as $\Pe\to 0$ and decay as $\Pe$ increases. However, we cannot determine the asymptotic decay rate of $\lb(\Pe)$ with any confidence because the range of \Peclet\ numbers spanned by our computations is too small. 

\begin{figure}
    \centering
    \includegraphics[width=\linewidth]{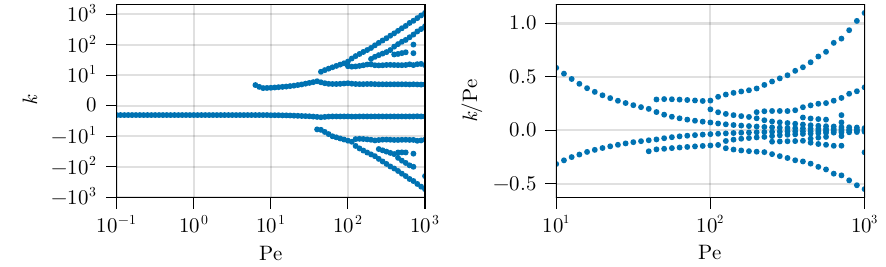}
    \caption{
    \emph{Left:} Active wavenumbers $k$ for problem \cref{e:ldp-annulus-symm-v2} as a function of $\Pe$. \emph{Right:} The same plot rescaled by $\Pe$, showing the superlinear growth of the largest active wavenumber. Both plots are obtained from discretizations of \cref{e:ldp-annulus-symm-v2} with $k\in\R$; similar plots for $k\in\Z$ are omitted for brevity.}
    \label{f:annulus-critical-k}
\end{figure}

\begin{figure}[t]
    \centering
    \includegraphics[scale=1.0]{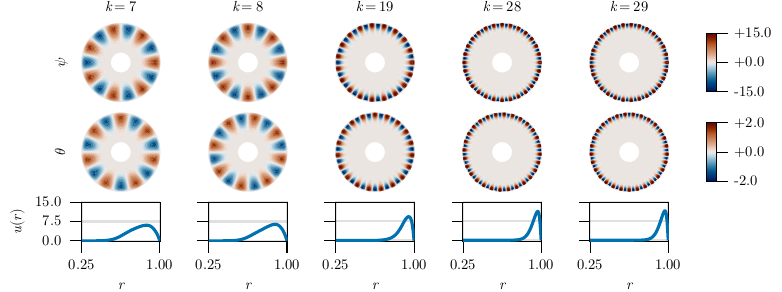}\\[10pt]
    \includegraphics[scale=1.0]{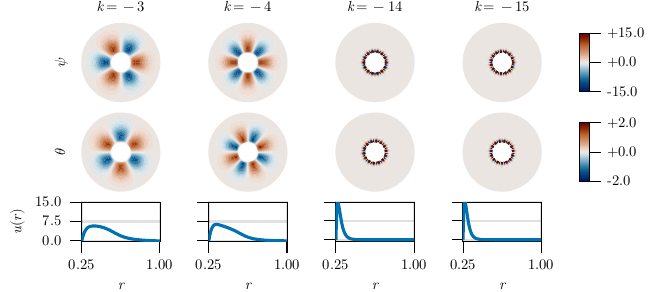}
    \caption{Critical streamfunction-temperature pairs $(\psi,\theta)$ and wavenumbers $k$ for which the constraints of problem \cref{e:ldp-annulus-symm} are active at $\Pe=100$. These pairs are recovered from functions $u(r)$, also shown, for which the constraints of problem \cref{e:ldp-annulus-symm-v2} are active.
    \label{f:annulus-critical-flows}
    }
\end{figure}

The main barrier to accessing larger \Peclet\ numbers is the high resolution needed to make our SDP discretizations accurate, which is due to two factors. First, the constraint $\mathcal{S}_k^+(\alpha,\xi) \succeq 0$ is active only for a small set of wavenumbers $k$ (cf. \cref{f:annulus-critical-k}), which are the zeros of the nonnegative function mapping $k$ to the principal eigenvalue of the linear operator associated with $\mathcal{S}_k^+(\alpha,\xi)$. However, the active set cannot be determined \emph{a priori}, so SDP discretizations must include a large number of LMIs. Second, even if one reduces this number using an `active set' method or an iterative strategy similar to that in \cite[\S5.2]{fpw2018}, the largest active wavenumbers appear to increase superlinearly with $\Pe$. Approximating the corresponding critical functions $u$ requires a high resolution because they correspond to flows with a high-frequency roll structure localized near the domain boundaries, as shown in \cref{f:annulus-critical-flows} for $\Pe=100$.

On the other hand, it is clear from our results that the optimal $\xi$ obtained computationally differs significantly from that used in the proof of \cref{th:analytical-lb}, which for the present annular geometry is a $\Pe$-dependent multiple of $\smash{(\frac{1-\rho^2}{\ln \rho}) \ln r + r^2 - 1}$. In the next section, we will prove that analytical choices of $\xi$ inspired by the profiles plotted in \cref{f:bounds_vs_Pe_annulus} can improve the bounds in \cref{th:analytical-lb} by a logarithmic factor.
\section{Improved analytical bounds for the annulus and the disk}
\label{s:improved-bounds}

Inspired by the numerical results presented in \cref{ss:annulus}, we now take a closer look at optimal cooling problems in circular and annular domains. By a suitable choice of coordinates, it suffices to consider the unit annulus $\Omega=\{\vec{x}\in\R^2: \rho < \abs{\vec{x}} < 1\}$ with inner radius $\rho$ and the unit disk $\Omega=\{\vec{x}\in\R^2: \abs{\vec{x}} < 1\}$. We will show that the optimal cooling efficiency bound $\efficiency(\Pe) \lesssim \Pe^{2}$ from \cref{th:analytical-lb} can be reduced by a logarithmic factor if the heat distribution $f$ is in $L^1(\Omega) \cap H^{-1}(\Omega)$ and has a uniformly positive average on any circle centered at the origin. This class includes the choice $f=1$ studied in \cref{ss:annulus}, as well as all uniformly positive, radially symmetric $f$.

Specifically, let $f$ be given in the polar coordinates $(r,\phi)$ for convenience, and let
\begin{equation*}
    \overline{f}(r) := \frac{1}{2\pi}
    \int_0^{2\pi} f(r,\phi) \, {\rm d}
    \phi
\end{equation*}
denote its azimuthal average.
We prove the following results, where we consider $\Pe > 1$ to simplify the notation. The restriction is mild because we are primarily interested in large $\Pe$ values.

\begin{theorem}
\label{th:refined-annulus}
    Let $\Omega=\{\vec{x}\in\R^2: \rho < \abs{\vec{x}} < 1\}$ be a unit annulus with inner radius $\rho \in(0,1)$. 
    If $f\in L^1(\Omega) \cap H^{-1}(\Omega)$ and $\overline{f}(r) \geq f_{\rm{min}} > 0$ for almost every $r \in (\rho,1)$, there exists a constant $C(\rho,f) \geq 16 \pi^2 \rho^6 (1-\rho)^3 f_{\rm min}^2 / (1-\rho + 8\pi \e \rho^4)$ such that
    \begin{equation*}
        \ocpval(\Pe) \geq 
        C(\rho,f) \frac{\ln^2 \Pe}{\Pe^2} \qquad \forall \, \Pe > 1.
    \end{equation*}
    Consequently, there exists a constant $C'(\rho,f)$ such that $\efficiency(\Pe) \leq C'(\rho,f) \frac{\Pe^2}{\ln^2 \Pe}$.
\end{theorem}

\begin{theorem}
\label{th:refined-disk}
    Let $\Omega$ be the unit disk. If $f\in L^1(\Omega) \cap H^{-1}(\Omega)$ and $\overline{f}(r) \geq f_{\rm{min}} > 0$ for almost every $r \in [0,1)$, there exists a constant $C(f) \geq \pi^2 f_{\rm min}^2 / (16 + 4 \pi \e)$ such that
    \begin{equation*}
        \ocpval(\Pe) \geq 
        C(f) \frac{\ln^2 \Pe}{\Pe^2} \qquad \forall \, \Pe > 1.
    \end{equation*}
    Consequently, there exists a constant $C'(f)$ such that $\efficiency(\Pe) \leq C'(f) \frac{\Pe^2}{\ln^2 \Pe}$.
\end{theorem}

\Cref{th:refined-annulus} is proved in \cref{ss:proof-annulus}. The proof of \cref{th:refined-disk} is very similar and relies on the same key estimates, so we only outline it in \cref{ss:proof-disk}.

\subsection{Proof of \texorpdfstring{\cref{th:refined-annulus}}{Theorem~\ref{th:refined-annulus}}}\label{ss:proof-annulus}

We only need to prove the lower bound on $\ocpval(\Pe)$, since the corresponding upper bound on $\efficiency(\Pe)$ follows immediately from \cref{e:efficiency}.
Restricting the maximization in the dual problem \cref{e:ldp} to functions $\xi$ that depend only on the radial position $r=|\vec{x}|$ gives the lower bound
\begin{equation}\label{e:annulus-general-f}
    \ocpval(\Pe) \geq 
    \lb(\Pe) \geq \max_{\substack{a \in \R\\ \xi \in H^1_0(\rho,1) \\ \QF(a,\xi)\succeq 0}} \;2\pi\int_{\rho}^1 \left( \xi \overline{f} - \frac14 |\xi'|^2 \right) r \dr - a. 
\end{equation}
Applying the same symmetry reduction arguments outlined in \cref{ss:annulus-symmetrization}, we conclude that the maximization problem on the right-hand side of this inequality can be rewritten as \cref{e:ldp-annulus-symm-v2} with $\overline{f}$ in place of $f$. We will construct admissible $\xi$ and $\alpha = \sqrt{a}$ for that problem, achieving an objective value no smaller than a multiple of $\smash{\frac{\ln^2 \Pe}{\Pe^2}}$.

\subsubsection{The candidate feasible point}
Let $\rho_0= \frac12(1+\rho)$ denote the average radius of the annulus. We claim that suitable choices of $\alpha$ and $\xi$ are
\begin{equation}\label{e:annulus-proof-xi-alpha}
    \alpha
    =
    \frac{\delta}{2\rho^2} \, \Pe
    \qquad\text{and}\qquad
    \xi(r) = \begin{cases}
        \delta \ln\left( \frac{r-\rho + \varepsilon}{\varepsilon}\right) & \text{if } r \in (\rho, \rho_0], \\
        \delta \ln\left( \frac{1 + \varepsilon - r}{\varepsilon}\right) & \text{if }r \in (\rho_0, 1),
    \end{cases}
\end{equation}
where $\delta$ and $\varepsilon$ are nonnegative $\Pe$-dependent constants given by
\begin{equation*}
    \varepsilon = \frac{1-\rho}{2 \e}\,\frac{1}{\Pe^2}
    \qquad\text{and}\qquad
    \delta
    =
    \frac{8 \pi f_{\rm min} \rho^5 (1-\rho)^2}{1-\rho + 8\pi \e \rho^4} \,
    \frac{\ln\Pe}{\Pe^2}.
\end{equation*}
\Cref{f:analytical-xi} shows the function $\xi$ for $\Pe=20$, $\rho=\frac14$ and $f_{\rm min} = 1$. Note the qualitative similarity with the numerically optimal $\xi$ obtained for $f=1$, shown in \cref{f:bounds_vs_Pe_annulus}.

\begin{figure}
    \centering
    \includegraphics[scale=1]{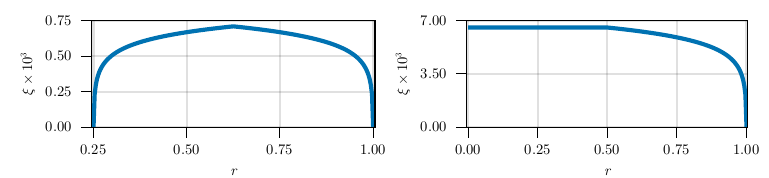}
    \caption{Analytical $\xi$ used to prove \cref{th:refined-annulus} (left) and \cref{th:refined-disk} (right). The left panel uses $\Pe=20$, $\rho=\frac14$, and $f_{\rm min} = 1$; the right one uses $\Pe=20$, $\rho_0=\frac12$, and $f_{\rm min} = 1$.}
    \label{f:analytical-xi}
\end{figure}

\subsubsection{Feasibility} 
We now verify that the function $\xi$ and the constant $\alpha$ in \cref{e:annulus-proof-xi-alpha} satisfy the constraints of problem \cref{e:ldp-annulus-symm-v2}, so they are admissible for the maximization problem in~\cref{e:annulus-general-f}. First, we establish a version of Hardy's inequality.

\begin{lemma}\label{lemma:hardy}
    Every $u\in H_{0}^{1}(\rho,1)$ satisfies
    \begin{align*}
        \int_{\rho}^{\rho_0}\frac{1}{(r-\rho)^{2}}\,\frac{u^{2}}{r}\,\dr	
        &\leq\frac{4}{\rho^{2}}\int_{\rho}^{\rho_0}r\abs{u'}^{2}\,\dr,
        \\
        \int_{\rho_0}^{1}\frac{1}{(1-r)^{2}}\,\frac{u^{2}}{r}\,\dr	
        &\leq\frac{4}{\rho^{2}}\int_{\rho_0}^{1}r\abs{u'}^{2}\,\dr.
    \end{align*}
\end{lemma}

\begin{proof}
    Since $u(\rho)=0$, we can use Hardy's inequality \cite[Exercise~8.8]{Brezis2011} to estimate 
    \begin{equation*}
    \int_{\rho}^{\rho_0}\frac{1}{(r-\rho)^{2}}\frac{u^{2}}{r}\,\dr
    \leq
    \frac{1}{\rho} \int_{\rho}^{\rho_0}\frac{u^{2}}{(r-\rho)^{2}}\,\dr
    \overset{\text{Hardy}}{\leq}
    \frac{4}{\rho} \int_{\rho}^{\rho_0}\abs{u'}^{2}\,\dr
    \leq
    \frac{4}{\rho^{2}} \int_{\rho}^{\rho_0}r\abs{u'}^{2}\,\dr.
    \end{equation*}
    This proves the first inequality. The second one follows similarly.
\end{proof}

Next, we use these Hardy-type estimates to establish the claimed feasibility.

\begin{proposition}\label{prop:annulus-feasibility}
    Let $\xi$ and $\alpha$ be as in \cref{e:annulus-proof-xi-alpha}. For all $u\in H^1_0(\rho,1)$ and all $k \in \Z$,
    \begin{equation*}
        \int_\rho^1 
        \alpha \left(r\abs{u'}^2+ \frac{k^2}{r}u^2 \right)
        + \frac{\Pe}2 \,k\, \xi' u^2 \, \dr \geq 0.
    \end{equation*}
\end{proposition}

\begin{proof}
    Since $\alpha= \Pe\,\delta / (2\rho^2)$ by definition, and since
    \begin{align*}
        \left|\int_{\rho}^{1}\frac{k}{2}\Pe\, \xi'u^{2}\,\dr\right| 
        \leq
        \frac{\Pe\, \delta}{2} \int_{\rho}^{\rho_0}\frac{\abs{k}u^{2}}{r-\rho+\varepsilon}\,\dr
        +\frac{\Pe \,\delta}{2}
        \int_{\rho_0}^{1}\frac{\abs{k} u^{2}}{1+\varepsilon-r}\,\dr,
    \end{align*}
    the claim follows if we can show that
    \begin{subequations}
        \begin{align}
        \label{e:annulus-cnstr-1}
        \int_{\rho}^{\rho_0}\frac{\abs{k} u^{2}}{r-\rho+\varepsilon}\,\dr 
        &\leq 
        \frac{1}{\rho^{2}} \int_{\rho}^{\rho_0}\left(r\abs{u'}^{2}+\frac{k^2}{r}u^{2}\right)\dr,
        \\
        \label{e:annulus-cnstr-2}
        \int_{\rho_0}^{1}
        \frac{\abs{k} u^{2}}{1+\varepsilon-r}\,\dr 
        &\leq 
        \frac{1}{\rho^{2}} \int_{\rho_0}^{1}
        \left(r\abs{u'}^{2}+\frac{k^2}{r}u^{2}\right)\dr.
    \end{align}
    \end{subequations}
    To prove \cref{e:annulus-cnstr-2}, we use the inequalities $\smash{\frac{1}{1+\varepsilon - r}}< \smash{\frac{1}{1-r}}$ and $1 < \smash{\frac1r}$ on $[\rho_0,1)$, the inequality $ab \leq \smash{\frac14}a^2+b^2$, \cref{lemma:hardy}, and the inequality $1<\smash{\frac{1}{\rho^2}}$ to estimate
    \begin{align*}
        \int_{\rho_0}^1 
        \frac{\abs{k} u^{2}}{1+\varepsilon-r}\,\dr 
        &\leq 
        \int_{\rho_0}^1\frac{\abs{k}}{1-r}\,u^{2}\,\dr
        \\
        &\leq
        \int_{\rho_0}^1 
        \left(\frac{1}{4(1-r)^{2}}+k^{2}\right)u^{2}\,\dr
        \\
        &\leq
        \int_{\rho_0}^1 
        \left(\frac{1}{4(1-r)^{2}}+k^{2}\right)\frac{u^{2}}{r}\,\dr
        \\
        &\leq
        \int_{\rho_0}^1
        \left(\frac{r}{\rho^{2}}\abs{u'}^{2}+\frac{k^{2}}{r}u^{2}\right)\,\dr
        \\
        &\leq
        \frac{1}{\rho^{2}}
        \int_{\rho_0}^1
        \left(r\abs{u'}^{2}+\frac{k^{2}}{r}u^{2}\right)\,\dr.
    \end{align*}
    Inequality \cref{e:annulus-cnstr-1} follows analogously using the first inequality in \cref{lemma:hardy}.
\end{proof}

\subsubsection{Estimating the objective value}
\label{ss:annulus-proof-objective}

Having verified that the function $\xi$ and the constant $\alpha$ in \cref{e:annulus-proof-xi-alpha} are admissible for the maximization problem in \cref{e:annulus-general-f}, we conclude the proof of \cref{th:refined-annulus} by finding a constant $C(\rho, f)$ such that
\begin{equation}\label{e:annulus-proof-cost-bound}
    2\pi\int_{\rho}^1 \left( \xi \overline{f} - \frac14 |\xi'|^2 \right) r \dr - \alpha^2 \geq 
    C(\rho, f)
    \frac{\ln^{2}\Pe}{\Pe^{2}}.
\end{equation}

Since $\xi$ is nonnegative and symmetric with respect to $r=\rho_0$ (see the left panel in \cref{f:analytical-xi}), we can estimate
\begin{align*}
    \int_{\rho}^1 \xi \overline{f}\, r\dr
    &\geq 
    2  f_{\rm min} \, \rho
    \int_{\rho_0}^1 \xi \dr
    \\
    &=
    2f_{\rm min} \, \delta \rho \left[ (1+\varepsilon-\rho_0) \ln\left( \frac{1+\varepsilon-\rho_0}{\varepsilon} \right) - (1-\rho_0)\right]
    \\
    &\geq 
    2f_{\rm min} \, \delta \rho (1-\rho_0) \left[ \ln\left( \frac{1-\rho_0}{\varepsilon} \right)  - 1\right]
    \\
    &=
    f_{\rm min} \, \delta \rho (1 - \rho) 
     \ln\left(\frac{1-\rho}{2 \e \varepsilon}\right).
\end{align*}
We also have
\begin{align*}
    \int_{\rho}^{1}|\xi'|^{2}r\,\dr 
    &\leq
    \int_{\rho}^{\rho_0}|\xi'|^{2}\,\dr+\int_{\rho_0}^{1}|\xi'|^{2}\,\dr
    \\
    &=
    \int_{\rho}^{\rho_0}\frac{\delta^{2}}{(r-\rho+\varepsilon)^{2}}\,\dr
    +
    \int_{\rho_0}^{1}\frac{\delta^{2}}{(1+\varepsilon-r)^{2}}\,\dr
    \\
    &=
    \delta^{2}\left(\frac{1}{\varepsilon}-\frac{1}{\rho_0-\rho+\varepsilon}+\frac{1}{\varepsilon}-\frac{1}{1+\varepsilon-\rho_0}\right)
    \leq
    \frac{2\delta^{2}}{\varepsilon}.
\end{align*}
Using these inequalities, as well as the definitions of $\alpha$, $\delta$, and $\varepsilon$, we arrive (after some lengthy but straightforward algebra omitted for brevity) at
\begin{align*}
    2\pi\int_{\rho}^1 \left( \xi \overline{f} - \frac14 |\xi'|^2 \right) r \dr - \alpha^2 
    &\geq 
    2\pi f_{\rm min} \, \delta \rho (1 - \rho) 
    \ln\left(\frac{1-\rho}{2 \e \varepsilon}\right)
    -
    \frac{\pi \delta^{2}}{\varepsilon}
    - \alpha^2
    \\
    &= \frac{16 \pi^2 \rho^6 (1-\rho)^3  f_{\rm min}^2}{(1-\rho) + 8\pi\e \rho^4}
    \,
    \frac{\ln^2\Pe}{\Pe^2}.
\end{align*}
This is precisely \cref{e:annulus-proof-cost-bound} with an explicit lower estimate for the constant $C(\rho, f)$.

\subsection{Proof of \texorpdfstring{\cref{th:refined-disk}}{Theorem~\ref{th:refined-disk}}}\label{ss:proof-disk}

We prove the lower bound on $\ocpval(\Pe)$; the upper bound on $\efficiency(\Pe)$ then follows from \cref{e:efficiency}. Fix an arbitrary $\rho_0 \in (0,1)$. Set
\begin{equation*}
    \varepsilon = \frac{1-\rho_0}{\e} \, \frac{1}{\Pe^2}
    \qquad\text{and}\qquad
    \delta
    =
    \frac{8 \pi f_{\rm min} \rho_0^5 (1-\rho_0)^2}{1-\rho_0 + 2\pi \e \rho_0^4} \,
    \frac{\ln\Pe}{\Pe^2}.
\end{equation*}
We claim that
\begin{equation*}
    \alpha = \frac{\delta}{2\rho_0^2} \, \Pe
    \qquad\text{and}\qquad
    \xi(\vec{x}) = \begin{cases}
        \delta \ln\left( \frac{1 + \varepsilon - \rho_0}{\varepsilon}\right) & \text{if } \abs{\vec{x}} \in [0,\rho_0] \\
        \delta \ln\left( \frac{1 + \varepsilon - \abs{\vec{x}}}{\varepsilon}\right) & \text{if }\abs{\vec{x}} \in (\rho_0, 1)
    \end{cases}
\end{equation*}
are admissible for the `balanced' version of problem \cref{e:ldp} described in \cref{rem:symmetrization-1} and attain an objective value no smaller than a multiple of $\smash{\frac{\ln^2 \Pe}{\Pe^2}}$. (For illustration, the function $\xi$ is plotted in \cref{f:analytical-xi} for $\Pe=20$, $f_{\rm min}=1$, and $\rho_0 = \frac12$.)

To verify this claim, let $B_{\rho_0}$ denote the disk of radius $\rho_0$ centered at the origin. Suppose for now that $\xi$ and $\alpha$ are admissible. Since $\xi$ depends only on the radial coordinate $r=\abs{\vec{x}}$ and is a positive constant in $B_{\rho_0}$,  the assumption on $f$ in \cref{th:refined-disk} implies that $\int_{B_{\rho_0}} f \xi \dx > 0$. Then, the objective value attained by $\xi$ and $\alpha$ satisfies
\begin{align*}
    \avg{f \xi - \frac14\abs{\nabla \xi}^2} - \alpha^2
    &\geq 
    \int_{\Omega \setminus B_{\rho_0}} {f \xi - \frac14\abs{\nabla \xi}^2} \,\dx - \alpha^2
    \\
    &
    =2\pi \int_{\rho_0}^1 \left( \overline{f} \xi - \frac14\abs{\xi'}^2 \right) r\dr - \alpha^2
    \\
    &\geq 2\pi  \rho_0(1-\rho_0) f_{\rm min} \delta \ln\left(\frac{1-\rho_0}{\e\varepsilon}\right) - \frac{\pi\delta^2}{2\varepsilon} - \alpha^2
    \\
    &=
    \frac{16\pi^2 \rho_0^6 (1-\rho_0)^3 f_{\rm min}^2 }{1-\rho_0 + 2\pi \e \rho_0^4} \frac{\ln^2 \Pe}{\Pe^2}
\end{align*}
where the second inequality follows from estimates similar to those used in \cref{ss:annulus-proof-objective} for the annulus. This lower bound is maximized when $\rho_0 \approx 0.5274$, so we set $\rho_0=\frac12$ to obtain \cref{th:refined-disk} with $C(f) \geq \pi^2 f_{\rm min}^2 / (16 + 4\pi\e)$.

It remains to verify that our choices of $\xi$ and $\alpha$ are admissible, meaning that the balanced quadratic form $\QF'(\alpha,\xi)[\psi,\theta]$ in \cref{e:qf-symmetrized} is nonnegative for all $\theta\in\Hoz$ and all $\psi \in \mathcal{H}$. Here, $\mathcal{H}=\Hoz$ because the disk $\Omega$ is simply connected. After integrating the $\xi$-dependent term in $\QF'(\alpha,\xi)[\psi,\theta]$ by parts and observing that $\xi$ is constant on $B_{\rho_0}$, we find that
\begin{align*}
    \QF'(\alpha,\xi)[\psi,\theta]
    &= 
    \avg{ \alpha |\nabla \psi|^2 + \alpha \abs{\nabla\theta}^2 + \Pe\,\theta \, \nabla^\perp \psi \cdot \nabla \xi}
    \\
    &\geq 
    \int_{\Omega \setminus B_{\rho_0}} { \alpha|\nabla \psi|^2 + \alpha \abs{\nabla\theta}^2 + \Pe\,\theta \, \nabla^\perp \psi \cdot \nabla \xi} \,\dx.
\end{align*}
Since the last integral is over the annulus $\Omega \setminus B_{\rho_0}$, we can analyze it with the same steps used for the annulus in \cref{ss:proof-annulus}, the only difference being a lack of boundary conditions on $\psi$ and $\theta$ when $\abs{\vec{x}}=\rho_0$. Specifically, we can perform the same Fourier series expansion described in \cref{app:fourier-annulus}, apply the symmetrization steps outlined in \cref{ss:annulus-symmetrization}, and conclude that $\QF'(\alpha,\xi)$ is positive semidefinite if
\begin{equation*}
    \int_{\rho_0}^1 
        \alpha \left(r|u'|^2+ \frac{k^2}{r}u^2 \right)
        + \frac{\Pe}2 \,k\, \xi' u^2 \, \dr \geq 0
\end{equation*}
for all $k \in \Z$ and all $u\in H^1(\rho_0,1)$ satisfying $u(1)=0$. One can verify that this inequality holds for the chosen $\xi$ and $\alpha$ with the same arguments used to prove \cref{prop:annulus-feasibility}, in particular using 
a version of \cref{e:annulus-cnstr-2} with $\rho_0$ replacing $\rho$, which holds for $u \in H^1(\rho_0,1)$ satisfying $u(1) = 0$ and has an identical proof. \Cref{th:refined-disk} is thus proved.
\section{Conclusion}\label{s:conclusion}

In this work, we used Lagrange duality and semidefinite programming to bound the maximum cooling efficiency of steady, energy-constrained incompressible flows in two dimensions. We measured efficiency via the inverse of the mean square gradient of the fluid's temperature $T$, so  optimal flows solve an optimal control problem minimizing $\smash{\int_\Omega \abs{\nabla T}^2 \dx}$. We showed that the Lagrange dual of this optimal control problem is a well-posed `infinite-dimensional SDP' that can be approximated with guaranteed convergence by finite-dimensional SDP discretizations. We demonstrated this on optimal cooling problems in a square and an annulus, showing also how problem symmetries can be exploited to pose SDPs with reduced complexity. Finally, we proved new analytical bounds on the optimal cooling efficiency $\efficiency(\Pe)$, where the square of the \Peclet\ number $\Pe$ measures the flow's kinetic energy. 
Specifically, we first extended to general domains the bound $\efficiency(\Pe)\lesssim\Pe^{2}$ proved for disks in \cite{Tobasco2022}. Then, inspired by computations based on our SDP framework, we established the improved upper bound $\efficiency(\Pe) \lesssim \Pe^{2} / \ln^2\Pe$ for flows in disks and annuli when the heat distribution has a uniformly positive azimuthal average. We do not know if this logarithmic correction is optimal because the range of \Peclet\ numbers spanned by our computations for the annulus in \cref{ss:annulus} is too small to estimate the asymptotic scaling of optimal bounds.

Although we focused only on steady, energy-constrained flows in heated domains with a cold boundary, our techniques can be applied more generally. Examples include optimal cooling problems with a mix of isothermal and insulating boundaries, with no-slip boundary conditions, or with bounds on the $H^s$ norm of the fluid's velocity for $s > 0$. (For $s=1$, one obtains the enstrophy-constrained flows studied in \cite{Tobasco2022, Song2023, Iyer2022, DoeringTobasco2019} since $\|\vec{u}\|_{H^1}$ and $\|\nabla \vec{u}\|_{L^2}$ are equivalent norms if $\vec{u}$ satisfies the no-penetration or no-slip boundary conditions.) Problems where the control cost $\int_\Omega |\nabla T|^2\,\dx$ is replaced by a generic bounded quadratic functional $\mathcal{C}(T)$ on $\Hoz$ can also be handled with little additional effort. Indeed, solutions to the advection-diffusion equation \cref{e:ade} satisfy the energy balance $\int_\Omega |\nabla T|^2 - fT \,\dx = 0$, so it suffices to bound the translated cost $\mathcal{C}(T) + b\int_\Omega |\nabla T|^2 - fT \,\dx$ for $b>0$ large enough to make this functional bounded below (see also \cite[\S2.1]{Tobasco2022} for a similar approach). One can also view $b$ as a Lagrange multiplier for the energy balance and optimize it within our SDP framework. Finally, our methods and the extensions just described generalize to time-dependent optimal cooling problems where the cooling efficiency is defined in a space-time averaged sense. There are additional technicalities in the definition of weak solutions to the advection-diffusion equation, but these can be handled as explained in \cite{Tobasco2022}.

On the other hand, various questions remain open. First, we have no convergence rate for the SDP discretizations of the dual problem \cref{e:ldp}. Our computations suggest that classical approximation rates for Galerkin methods should apply. Proving this, however, requires more than $H^1$ regularity for the optimal dual variable $\xi$ and for the critical functions associated with the constraint of \cref{e:ldp}. One way to show higher regularity might be to study the optimality conditions for \cref{e:ldp}. 

A second open challenge is to remove the computational bottlenecks in the solution of high-resolution SDP discretizations of \cref{e:ldp}.
One option is to exploit structural properties of \cref{e:ldp} and its discretizations beyond symmetries. For example, the small number of `critical functions' observed in our computations suggests that the dual of the SDP \cref{e:sdp} generically has a low-rank solution. SDP solvers exploiting this low-rank structure (see, e.g., \cite{Burer2003,Burer2006,Bellavia2021,Habibi2023,Souto2022,Monteiro2026}) may thus be able to accurately approximate $\lb(\Pe)$ over a large enough range of \Peclet\ numbers to estimate the large-$\Pe$ asymptotics. Another option is to replace general-purpose algorithms for finite-dimensional SDPs with function-space algorithms tailored to the infinite-dimensional problem \cref{e:ldp}. Unfortunately, we are not aware of any such algorithms beyond the gradient methods developed in \cite{Wen2013,Wen2015,Ding2019,Wen2022} for problems arising when bounding the time-averaged dissipation of turbulent flows, which can be adapted to \cref{e:ldp} but do not have general convergence guarantees. 

Finally, since the optimal cooling problem \cref{e:ocp} is not convex, the bound $\lb(\Pe)$ obtained from its Lagrange dual \cref{e:ldp} need not be sharp. Non-sharp bounds may still reproduce the large-$\Pe$ asymptotics of the optimal cost $\ocpval(\Pe)$, but we know of no nontrivial case where this happens because known upper and lower bounds on $\ocpval(\Pe)$ exhibit different asymptotics. (Of course, $\ocpval(\Pe)=\lb(\Pe)=0$ for all $\Pe$ if the domain is not heated.) If duality gaps affect the large-$\Pe$ asymptotics, or if quantitatively sharp bounds are required, then one must go beyond Lagrange duality. For discretizations of \cref{e:ocp}, which are quadratic programs, the moment-sum-of-squares hierarchy \cite{Laurent2009,Lasserre2010,Lasserre2024} provides a sequence of convex dual problems that generalize the Lagrange dual and converge if certain `algebraic compactness' conditions hold. In principle, it should be possible to formulate an infinite-dimensional version of this hierarchy by adapting ideas proposed in \cite{Goulart2012,Chernyshenko2014a,Goluskin2019,HenrionRudi2025,henrion2023infmom}. While current computational bottlenecks are likely to make a naive application of this approach prohibitively expensive, the particular structure of problem \cref{e:ocp} may allow for some progress. How much can be achieved remains to be seen, but it is an avenue worth exploring.

% Appendix
\appendix
\crefalias{section}{appendix}
\section{Fourier expansion of \texorpdfstring{\cref{e:ldp}}{(\ref{e:lpd})} in an annulus} \label{app:fourier-annulus}

Consider problem \cref{e:ldp} in the annulus $\Omega = \smash{\{\vec{x}:\;\rho < \abs{\vec{x}} < 1\}}$. We work in polar coordinates $(r,\phi)$ and assume that $f=f(r)$ depends only on the radial direction. We can then take $\xi=\xi(r)$, too, so the objective function in \cref{e:ldp} reduces to that in~\cref{e:ldp-annulus-symm}. 

We claim that the constraints of these two problems are also equivalent. To show this, we introduce the Fourier series expansions
\begin{equation*}
    \psi(r,\phi) = \sum_{k\in\Z} \hat{\psi}_k(r) e^{ik\phi}
    \quad \text{and} \quad
    \theta(r,\phi) = \sum_{k\in\Z} \hat{\theta}_k(r) e^{ik\phi},
\end{equation*}
where the complex-valued functions $\smash{\hat{\psi}_k}$ and $\smash{\hat{\theta}_k}$ satisfy $\smash{\hat{\psi}_{-k}=\hat{\psi}_k^*}$ and $\smash{\hat{\theta}_{-k}=\hat{\theta}_k^*}$ because $\psi$ and $\theta$ are real-valued. The boundary conditions on $\psi$ and $\theta$ require $\smash{\hat{\theta}_0(\rho)}=0$, $\smash{\hat{\psi}_0(1)} = 0$, $\smash{\hat{\theta}_0(1)}=0$ and
$\smash{\hat{\psi}_k(\rho) = 0}$, $\smash{\hat{\theta}_k(\rho)}=0$, $\smash{\hat{\psi}_k(1)} = 0$, and $\smash{\hat{\theta}_k(1)}=0$ for all $k \neq 0$.
In particular, $\smash{\hat{\psi}_k,\hat{\theta}_k \in H^1_0(\rho,1)}$ for all $k \neq 0$.

Next, we compute the Fourier expansion of the quadratic form $\QF(a,\xi)[\psi,\theta]$ from \cref{e:qf}. Using the orthogonality of the Fourier modes, we find that
\begin{equation*}
    \frac{1}{2\pi}\int_\Omega |\nabla \psi|^2 \dx =
    \sum_{k\in \Z} \int_\rho^1 \left( |\hat{\psi}_k'|^2  + \frac{k^2}{r^2} |\hat{\psi}_k|^2 \right) r\dr.
\end{equation*}
A similar identity holds for $\int |\nabla \theta|^2 \dx$. For the term $\int \xi \nabla^\perp \psi \cdot \nabla \theta\,\dx$, instead, we first integrate by parts and then expand in Fourier modes to find that
\begin{equation*}
    \frac{1}{2\pi}\avg{\xi \nabla^\perp \psi \cdot \nabla \theta}
    = 
    -\frac{1}{2\pi}\avg{\theta \nabla^\perp \psi \cdot \nabla \xi}
    = 
    \sum_{k\in \Z}  \int_\rho^1  ik \xi' {\hat{\psi}}_k^* \hat{\theta}_{k} \dr.
\end{equation*}
Combining all expansions we obtain
\begin{equation*}
    \frac{1}{2\pi}\QF(a,\xi)[\psi,\theta]
    =\sum_{k \in \Z} 
    \int_\rho^1
    a \left( r|\hat{\psi}_k'|^2  + \frac{k^2}{r} |\hat{\psi}_k|^2 \right)
    +  r|\hat{\theta}_k'|^2  + \frac{k^2}{r} |\hat{\theta}_k|^2 
    - ik \Pe\, \xi' {\hat{\psi}}_k^* \hat{\theta}_{k} \dr.
\end{equation*}

Now, for $k \neq 0$ write $\hat{\psi}_k = A_k - iB_k$ and $\hat{\theta}_k = C_k+iD_k$ for real-valued functions $A_k$, $B_k$, $C_k$, and $D_k$ in $H^1_0(\rho,1)$. A straightforward calculation yields
\begin{align*}
    \frac{1}{2\pi}\QF(a,\xi)[\psi,\theta]
    &=
    \QF_0(a,\xi)[\hat{\psi}_0,\hat{\theta}_0]
    +\sum_{\substack{k \in \Z \\ k\neq 0}} 
    \QF_k(a,\xi)[A_k,D_k] +
    \QF_k(a,\xi)[B_k,C_k],
\end{align*}
where $\QF_k(a,\xi)$ is defined by the formula in \cref{e:qf-fourier-annulus}.
Since $\hat{\psi}_0$, $\hat{\theta}_0$, and the functions $A_k$, $B_k$, $C_k$ and $D_k$ for different $\abs{k}$ can be chosen independently, and since $\QF_0(a,\xi)$ is positive semidefinite, the condition $\QF(a,\xi)\succeq 0$ holds on $\mathcal{H}\times\Hoz$ if and only if $\QF_k(a,\xi) \succeq 0$ on $H^1_0(\rho,1) \times H^1_0(\rho,1)$ for all nonzero $k\in \Z$. This proves our claim that the constraints of \cref{e:ldp} and \cref{e:ldp-annulus-symm} are equivalent when $\Omega$ is an annulus and the heat distribution $f$ is radially symmetric.

% Acknowledgments
\subsection*{Acknowledgments}
We thank David Goluskin, Andrew Wynn, John Craske, and Defne Ege Ozan for their feedback on the first version of this manuscript. We acknowledge the use of AI when proofreading the manuscript.

% Bibliography
\bibliographystyle{abbrvnat}
\bibliography{reflist}

@article {Farid2026,
    AUTHOR = {Rajkotia-Zaheer, Farid and Goluskin, David},
     TITLE = {Bounds on dissipation in three-dimensional planar shear flows: reduction to two-dimensional problems},
   JOURNAL = {J. Fluid Mech.},
  FJOURNAL = {Journal of Fluid Mechanics},
    VOLUME = {1026},
      YEAR = {2026},
     PAGES = {A8},
      hISSN = {0022-1120,1469-7645},
       hDOI = {10.1017/jfm.2025.10958},
       hURL = {https://doi.org/10.1017/jfm.2025.10958},
}

@article {Goulart2012,
    AUTHOR = {Goulart, Paul J. and Chernyshenko, Sergei I.},
     TITLE = {Global stability analysis of fluid flows using sum-of-squares},
   JOURNAL = {Phys. D},
  FJOURNAL = {Physica D. Nonlinear Phenomena},
    VOLUME = {241},
      YEAR = {2012},
    NUMBER = {6},
     PAGES = {692--704},
      hISSN = {0167-2789,1872-8022},
       hDOI = {10.1016/j.physd.2011.12.008},
       hURL = {https://doi.org/10.1016/j.physd.2011.12.008},
}

@article {Chernyshenko2014a,
    AUTHOR = {Chernyshenko, Sergei I. and Goulart, Paul J. and Huang, D. and Papachristodoulou, Antonis},
     TITLE = {Polynomial sum of squares in fluid dynamics: a review with a look ahead},
   JOURNAL = {Philos. Trans. Roy. Soc. A},
  FJOURNAL = {Philosophical Transactions of the Royal Society A. Mathematical, Physical and Engineering Sciences},
    VOLUME = {372},
      YEAR = {2014},
    NUMBER = {2020},
     PAGES = {20130350},
      hISSN = {1364-503X,1471-2962},
       hDOI = {10.1098/rsta.2013.0350},
       hURL = {https://doi.org/10.1098/rsta.2013.0350},
}

@article {Goluskin2019,
    AUTHOR = {Goluskin, David and Fantuzzi, Giovanni},
     TITLE = {Bounds on mean energy in the {K}uramoto--{S}ivashinsky equation computed using semidefinite programming},
   JOURNAL = {Nonlinearity},
  FJOURNAL = {Nonlinearity},
    VOLUME = {32},
      YEAR = {2019},
    NUMBER = {5},
     PAGES = {1705--1730},
      hISSN = {0951-7715,1361-6544},
       hDOI = {10.1088/1361-6544/ab018b},
       hURL = {https://doi.org/10.1088/1361-6544/ab018b},
}

@article {HenrionRudi2025,
    AUTHOR = {Henrion, Didier and Rudi, Alessandro},
     TITLE = {Solving moment and polynomial optimization problems on {S}obolev spaces},
   JOURNAL = {SIAM J. Optim.},
  FJOURNAL = {SIAM Journal on Optimization},
    VOLUME = {35},
      YEAR = {2025},
    NUMBER = {2},
     PAGES = {989--1003},
      hISSN = {1052-6234,1095-7189},
       hDOI = {10.1137/24M163133X},
       hURL = {https://doi.org/10.1137/24M163133X},
}

@misc {henrion2023infmom,
    AUTHOR = {Henrion, Didier and Infusino, Maria and Kuhlmann, Salma and Vinnikov, Victor},
     TITLE = {Infinite-dimensional moment-{SOS} hierarchy for nonlinear partial differential equations},
      YEAR = {2023},
HOWPUBLISHED = {arXiv:2305.18768 [math.OC]},
      hURL = {https://arxiv.org/abs/2305.18768},
}

@incollection {Parrilo2013,
    AUTHOR = {Parrilo, Pablo A.},
     TITLE = {Polynomial optimization, sums of squares, and applications},
 BOOKTITLE = {Semidefinite optimization and convex algebraic geometry},
    SERIES = {MOS-SIAM Ser. Optim.},
    VOLUME = {13},
     PAGES = {47--157},
 PUBLISHER = {SIAM, Philadelphia, PA},
      YEAR = {2013},
     hISBN = {978-1-611972-28-3},
       hDOI = {10.1137/1.9781611972290.ch3},
       hURL = {https://doi.org/10.1137/1.9781611972290.ch3},
}

@article {Chanillo2025,
    AUTHOR = {Chanillo, Sagun and Malchiodi, Andrea},
     TITLE = {Sharp bounds on the {N}usselt number in {R}ayleigh--{B}\'{e}nard convection and a bilinear estimate by {C}oifman--{M}eyer},
   JOURNAL = {Invent. Math.},
  FJOURNAL = {Inventiones Mathematicae},
    VOLUME = {240},
      YEAR = {2025},
    NUMBER = {2},
     PAGES = {633--660},
      hISSN = {0020-9910,1432-1297},
       hDOI = {10.1007/s00222-025-01326-z},
       hURL = {https://doi.org/10.1007/s00222-025-01326-z},
}

@article {fpw2018,
    AUTHOR = {Fantuzzi, Giovanni and Pershin, Anton and Wynn, Andrew},
     TITLE = {Bounds on heat transfer for {B}\'{e}nard--{M}arangoni convection at infinite {P}randtl number},
   JOURNAL = {J. Fluid Mech.},
  FJOURNAL = {Journal of Fluid Mechanics},
    VOLUME = {837},
      YEAR = {2018},
     PAGES = {562--596},
      hISSN = {0022-1120,1469-7645},
       hDOI = {10.1017/jfm.2017.858},
       hURL = {https://doi.org/10.1017/jfm.2017.858},
}

@article {Constantin1995,
    AUTHOR = {Constantin, Peter and Doering, Charles R.},
     TITLE = {Variational bounds on energy dissipation in incompressible flows. {II}. {C}hannel flow},
   JOURNAL = {Phys. Rev. E},
  FJOURNAL = {Physical Review E},
    VOLUME = {51},
      YEAR = {1995},
    NUMBER = {4},
     PAGES = {3192--3198},
      hISSN = {1063-651X,1095-3787},
       hDOI = {10.1103/PhysRevE.51.3192},
       hURL = {https://doi.org/10.1103/PhysRevE.51.3192},
}

@article {Doering1994,
    AUTHOR = {Doering, Charles R. and Constantin, Peter},
     TITLE = {Variational bounds on energy dissipation in incompressible flows: shear flow},
   JOURNAL = {Phys. Rev. E},
  FJOURNAL = {Physical Review E},
    VOLUME = {49},
      YEAR = {1994},
    NUMBER = {5},
     PAGES = {4087--4099},
      hISSN = {1063-651X,1095-3787},
       hDOI = {10.1103/PhysRevE.49.4087},
       hURL = {https://doi.org/10.1103/PhysRevE.49.4087},
}

@article {Doering1996,
    AUTHOR = {Doering, Charles R. and Constantin, Peter},
     TITLE = {Variational bounds on energy dissipation in incompressible flows. {III}. {C}onvection},
   JOURNAL = {Phys. Rev. E},
  FJOURNAL = {Physical Review E},
    VOLUME = {53},
      YEAR = {1996},
    NUMBER = {6},
     PAGES = {5957--5981},
      hISSN = {1063-651X,1095-3787},
       hDOI = {10.1103/PhysRevE.53.5957},
       hURL = {https://doi.org/10.1103/PhysRevE.53.5957},
}

@article {Wen2013,
    AUTHOR = {Wen, Baole and Chini, Gregory P. and Dianati, Navid and Doering, Charles R.},
     TITLE = {Computational approaches to aspect-ratio-dependent upper bounds and heat flux in porous medium convection},
   JOURNAL = {Phys. Lett. A},
  FJOURNAL = {Physics Letters. A},
    VOLUME = {377},
      YEAR = {2013},
    NUMBER = {41},
     PAGES = {2931--2938},
      hISSN = {0375-9601},
       hDOI = {10.1016/j.physleta.2013.09.009},
       hURL = {https://doi.org/10.1016/j.physleta.2013.09.009},
}

@article {Wen2015,
    AUTHOR = {Wen, Baole and Chini, Gregory P. and Kerswell, Rich R. and Doering, Charles R.},
     TITLE = {Time-stepping approach for solving upper-bound problems: application to two-dimensional {R}ayleigh--{B}\'{e}nard convection},
   JOURNAL = {Phys. Rev. E},
  FJOURNAL = {Physical Review E},
    VOLUME = {92},
      YEAR = {2015},
    NUMBER = {4},
     PAGES = {043012},
      hISSN = {1539-3755,1550-2376},
       hDOI = {10.1103/PhysRevE.92.043012},
       hURL = {https://doi.org/10.1103/PhysRevE.92.043012},
}

@article {Wen2022,
    AUTHOR = {Wen, Baole and Ding, Zijing and Chini, Gregory P. and Kerswell, Rich R.},
     TITLE = {Heat transport in {R}ayleigh--{B}\'{e}nard convection with linear marginality},
   JOURNAL = {Philos. Trans. Roy. Soc. A},
  FJOURNAL = {Philosophical Transactions of the Royal Society A. Mathematical, Physical and Engineering Sciences},
    VOLUME = {380},
      YEAR = {2022},
    NUMBER = {2225},
     PAGES = {20210039},
      hISSN = {1364-503X,1471-2962},
       hDOI = {10.1098/rsta.2021.0039},
       hURL = {https://doi.org/10.1098/rsta.2021.0039},
}

@article {Ding2019,
    AUTHOR = {Ding, Zijing and Marensi, Elena},
     TITLE = {Upper bound on angular momentum transport in {T}aylor--{C}ouette flow},
   JOURNAL = {Phys. Rev. E},
  FJOURNAL = {Physical Review E},
    VOLUME = {100},
      YEAR = {2019},
    NUMBER = {6},
     PAGES = {063109},
      hISSN = {2470-0045,2470-0053},
       hDOI = {10.1103/PhysRevE.100.063109},
       hURL = {https://doi.org/10.1103/PhysRevE.100.063109},
}

@article {Plasting2003,
    AUTHOR = {Plasting, S. C. and Kerswell, R. R.},
     TITLE = {Improved upper bound on the energy dissipation rate in plane {C}ouette flow: the full solution to {B}usse's problem and the {C}onstantin--{D}oering--{H}opf problem with one-dimensional background field},
   JOURNAL = {J. Fluid Mech.},
  FJOURNAL = {Journal of Fluid Mechanics},
    VOLUME = {477},
      YEAR = {2003},
     PAGES = {363--379},
      hISSN = {0022-1120,1469-7645},
       hDOI = {10.1017/S0022112002003361},
       hURL = {https://doi.org/10.1017/S0022112002003361},
}

@article {jump,
    AUTHOR = {Lubin, Miles and Dowson, Oscar and Garcia, Joaquim Dias and Huchette, Joey and Legat, Beno\^{\i}t and Vielma, Juan Pablo},
     TITLE = {Ju{MP} 1.0: recent improvements to a modeling language for mathematical optimization},
   JOURNAL = {Math. Program. Comput.},
  FJOURNAL = {Mathematical Programming Computation},
    VOLUME = {15},
      YEAR = {2023},
    NUMBER = {3},
     PAGES = {581--589},
      hISSN = {1867-2949,1867-2957},
       hDOI = {10.1007/s12532-023-00239-3},
       hURL = {https://doi.org/10.1007/s12532-023-00239-3},
}

@article {Alben2017a,
    AUTHOR = {Alben, S.},
     TITLE = {Optimal convection cooling flows in general 2{D} geometries},
   JOURNAL = {J. Fluid Mech.},
  FJOURNAL = {Journal of Fluid Mechanics},
    VOLUME = {814},
      YEAR = {2017},
     PAGES = {484--509},
      hISSN = {0022-1120,1469-7645},
       hDOI = {10.1017/jfm.2017.35},
       hURL = {https://doi.org/10.1017/jfm.2017.35},
}

@article {Alben2017b,
    AUTHOR = {Alben, S.},
     TITLE = {Improved convection cooling in steady channel flows},
   JOURNAL = {Phys. Rev. Fluids},
  FJOURNAL = {Physical Review Fluids},
    VOLUME = {2},
      YEAR = {2017},
    NUMBER = {10},
     PAGES = {104501},
      hISSN = {2469-990X},
       hDOI = {10.1103/PhysRevFluids.2.104501},
       hURL = {https://doi.org/10.1103/PhysRevFluids.2.104501},
}

@article {Marcotte2018,
    AUTHOR = {Marcotte, Florence and Doering, Charles R. and Thiffeault, Jean-Luc and Young, William R.},
     TITLE = {Optimal heat transfer and optimal exit times},
   JOURNAL = {SIAM J. Appl. Math.},
  FJOURNAL = {SIAM Journal on Applied Mathematics},
    VOLUME = {78},
      YEAR = {2018},
    NUMBER = {1},
     PAGES = {591--608},
      hISSN = {0036-1399,1095-712X},
       hDOI = {10.1137/17M1150220},
       hURL = {https://doi.org/10.1137/17M1150220},
}

@article {Iyer2022,
    AUTHOR = {Iyer, Gautam and Van, Truong-Son},
     TITLE = {Bounds on the heat transfer rate via passive advection},
   JOURNAL = {SIAM J. Math. Anal.},
  FJOURNAL = {SIAM Journal on Mathematical Analysis},
    VOLUME = {54},
      YEAR = {2022},
    NUMBER = {2},
     PAGES = {1927--1965},
      hISSN = {0036-1410,1095-7154},
       hDOI = {10.1137/21M1394497},
       hURL = {https://doi.org/10.1137/21M1394497},
}

@article {Fantuzzi2022,
    AUTHOR = {Fantuzzi, Giovanni and Arslan, Ali and Wynn, Andrew},
     TITLE = {The background method: theory and computations},
   JOURNAL = {Philos. Trans. Roy. Soc. A},
  FJOURNAL = {Philosophical Transactions of the Royal Society A. Mathematical, Physical and Engineering Sciences},
    VOLUME = {380},
      YEAR = {2022},
    NUMBER = {2225},
     PAGES = {20210038},
      hISSN = {1364-503X,1471-2962},
       hDOI = {10.1098/rsta.2021.0038},
       hURL = {https://doi.org/10.1098/rsta.2021.0038},
}

@article {Tobasco2022,
    AUTHOR = {Tobasco, Ian},
     TITLE = {Optimal cooling of an internally heated disc},
   JOURNAL = {Philos. Trans. Roy. Soc. A},
  FJOURNAL = {Philosophical Transactions of the Royal Society A. Mathematical, Physical and Engineering Sciences},
    VOLUME = {380},
      YEAR = {2022},
    NUMBER = {2225},
     PAGES = {20210040},
      hISSN = {1364-503X,1471-2962},
       hDOI = {10.1098/rsta.2021.0040},
       hURL = {https://doi.org/10.1098/rsta.2021.0040},
}

@article {TobascoDoering2017,
    AUTHOR = {Tobasco, Ian and Doering, Charles R.},
     TITLE = {Optimal wall-to-wall transport by incompressible flows},
   JOURNAL = {Phys. Rev. Lett.},
  FJOURNAL = {Physical Review Letters},
    VOLUME = {118},
      YEAR = {2017},
    NUMBER = {26},
     PAGES = {264502},
      hISSN = {0031-9007,1079-7114},
       hDOI = {10.1103/PhysRevLett.118.264502},
       hURL = {https://doi.org/10.1103/PhysRevLett.118.264502},
}

@article {DoeringTobasco2019,
    AUTHOR = {Doering, Charles R. and Tobasco, Ian},
     TITLE = {On the optimal design of wall-to-wall heat transport},
   JOURNAL = {Comm. Pure Appl. Math.},
  FJOURNAL = {Communications on Pure and Applied Mathematics},
    VOLUME = {72},
      YEAR = {2019},
    NUMBER = {11},
     PAGES = {2385--2448},
      hISSN = {0010-3640,1097-0312},
       hDOI = {10.1002/cpa.21832},
       hURL = {https://doi.org/10.1002/cpa.21832},
}

@article {Kumar2024,
    AUTHOR = {Kumar, Anuj},
     TITLE = {Three dimensional branching pipe flows for optimal scalar transport between walls},
   JOURNAL = {Nonlinearity},
  FJOURNAL = {Nonlinearity},
    VOLUME = {37},
      YEAR = {2024},
    NUMBER = {11},
     PAGES = {115011},
      hISSN = {0951-7715,1361-6544},
       hDOI = {10.1088/1361-6544/ad789e},
       hURL = {https://doi.org/10.1088/1361-6544/ad789e},
}

@article {Hassanzadeh2014,
    AUTHOR = {Hassanzadeh, Pedram and Chini, Gregory P. and Doering, Charles R.},
     TITLE = {Wall to wall optimal transport},
   JOURNAL = {J. Fluid Mech.},
  FJOURNAL = {Journal of Fluid Mechanics},
    VOLUME = {751},
      YEAR = {2014},
     PAGES = {627--662},
      hISSN = {0022-1120,1469-7645},
       hDOI = {10.1017/jfm.2014.306},
       hURL = {https://doi.org/10.1017/jfm.2014.306},
}

@article {Motoki2018,
    AUTHOR = {Motoki, Shingo and Kawahara, Genta and Shimizu, Masaki},
     TITLE = {Maximal heat transfer between two parallel plates},
   JOURNAL = {J. Fluid Mech.},
  FJOURNAL = {Journal of Fluid Mechanics},
    VOLUME = {851},
      YEAR = {2018},
     PAGES = {R4},
      hISSN = {0022-1120,1469-7645},
       hDOI = {10.1017/jfm.2018.557},
       hURL = {https://doi.org/10.1017/jfm.2018.557},
}

@article {Souza2020,
    AUTHOR = {Souza, Andre N. and Tobasco, Ian and Doering, Charles R.},
     TITLE = {Wall-to-wall optimal transport in two dimensions},
   JOURNAL = {J. Fluid Mech.},
  FJOURNAL = {Journal of Fluid Mechanics},
    VOLUME = {889},
      YEAR = {2020},
     PAGES = {A34},
      hISSN = {0022-1120,1469-7645},
       hDOI = {10.1017/jfm.2020.42},
       hURL = {https://doi.org/10.1017/jfm.2020.42},
}

@article {Shaw2007,
    AUTHOR = {Shaw, Tiffany A. and Thiffeault, Jean-Luc and Doering, Charles R.},
     TITLE = {Stirring up trouble: multi-scale mixing measures for steady scalar sources},
   JOURNAL = {Phys. D},
  FJOURNAL = {Physica D. Nonlinear Phenomena},
    VOLUME = {231},
      YEAR = {2007},
    NUMBER = {2},
     PAGES = {143--164},
      hISSN = {0167-2789,1872-8022},
       hDOI = {10.1016/j.physd.2007.05.001},
       hURL = {https://doi.org/10.1016/j.physd.2007.05.001},
}

@article {DoeringThiffeault2006,
    AUTHOR = {Doering, Charles R. and Thiffeault, Jean-Luc},
     TITLE = {Multiscale mixing efficiencies for steady sources},
   JOURNAL = {Phys. Rev. E},
  FJOURNAL = {Physical Review E},
    VOLUME = {74},
      YEAR = {2006},
    NUMBER = {2},
     PAGES = {025301},
      hISSN = {1539-3755,1550-2376},
       hDOI = {10.1103/PhysRevE.74.025301},
       hURL = {https://doi.org/10.1103/PhysRevE.74.025301},
}

@article {Thiffeault2012,
    AUTHOR = {Thiffeault, Jean-Luc},
     TITLE = {Using multiscale norms to quantify mixing and transport},
   JOURNAL = {Nonlinearity},
  FJOURNAL = {Nonlinearity},
    VOLUME = {25},
      YEAR = {2012},
    NUMBER = {2},
     PAGES = {R1--R44},
      hISSN = {0951-7715,1361-6544},
       hDOI = {10.1088/0951-7715/25/2/R1},
       hURL = {https://doi.org/10.1088/0951-7715/25/2/R1},
}

@article {Thiffeault2004,
    AUTHOR = {Thiffeault, Jean-Luc and Doering, Charles R. and Gibbon, John D.},
     TITLE = {A bound on mixing efficiency for the advection-diffusion equation},
   JOURNAL = {J. Fluid Mech.},
  FJOURNAL = {Journal of Fluid Mechanics},
    VOLUME = {521},
      YEAR = {2004},
     PAGES = {105--114},
      hISSN = {0022-1120,1469-7645},
       hDOI = {10.1017/S0022112004001739},
       hURL = {https://doi.org/10.1017/S0022112004001739},
}

@article {Song2023,
    AUTHOR = {Song, Binglin and Fantuzzi, Giovanni and Tobasco, Ian},
     TITLE = {Bounds on heat transfer by incompressible flows between balanced sources and sinks},
   JOURNAL = {Phys. D},
  FJOURNAL = {Physica D. Nonlinear Phenomena},
    VOLUME = {444},
      YEAR = {2023},
     PAGES = {133591},
      hISSN = {0167-2789,1872-8022},
       hDOI = {10.1016/j.physd.2022.133591},
       hURL = {https://doi.org/10.1016/j.physd.2022.133591},
}

@article {Shor1987,
    AUTHOR = {Shor, N. Z.},
     TITLE = {Quadratic optimization problems},
   JOURNAL = {Izv. Akad. Nauk SSSR Tekhn. Kibernet.},
  FJOURNAL = {Izvestiya Akademii Nauk SSSR. Tekhnicheskaya Kibernetika},
      YEAR = {1987},
    NUMBER = {1},
     PAGES = {128--139},
      hISSN = {0002-3388},
}

@incollection {Nesterov2000,
    AUTHOR = {Nesterov, Yuri and Wolkowicz, Henry and Ye, Yinyu},
     TITLE = {Semidefinite programming relaxations of nonconvex quadratic optimization},
 BOOKTITLE = {Handbook of semidefinite programming},
    SERIES = {Internat. Ser. Oper. Res. Management Sci.},
    VOLUME = {27},
     PAGES = {361--419},
 PUBLISHER = {Kluwer Acad. Publ., Boston, MA},
      YEAR = {2000},
     hISBN = {0-7923-7771-0},
       hDOI = {10.1007/978-1-4615-4381-7_13},
       hURL = {https://doi.org/10.1007/978-1-4615-4381-7_13},
}

@article {Kerswell1998,
    AUTHOR = {Kerswell, R. R.},
     TITLE = {Unification of variational principles for turbulent shear flows: the background method of {D}oering--{C}onstantin and the mean-fluctuation formulation of {H}oward--{B}usse},
   JOURNAL = {Phys. D},
  FJOURNAL = {Physica D. Nonlinear Phenomena},
    VOLUME = {121},
      YEAR = {1998},
    NUMBER = {1--2},
     PAGES = {175--192},
      hISSN = {0167-2789,1872-8022},
       hDOI = {10.1016/S0167-2789(98)00104-3},
       hURL = {https://doi.org/10.1016/S0167-2789(98)00104-3},
}

@article {Kerswell1999,
    AUTHOR = {Kerswell, R. R.},
     TITLE = {Variational principle for the {N}avier--{S}tokes equations},
   JOURNAL = {Phys. Rev. E},
  FJOURNAL = {Physical Review E},
    VOLUME = {59},
      YEAR = {1999},
    NUMBER = {5},
     PAGES = {5482--5494},
      hISSN = {1063-651X,1095-3787},
       hDOI = {10.1103/PhysRevE.59.5482},
       hURL = {https://doi.org/10.1103/PhysRevE.59.5482},
}

@article {Gertler2025,
    AUTHOR = {Gertler, Shai and Kuang, Zeyu and Christie, Colin and Li, Hao and Miller, Owen D.},
     TITLE = {Many photonic design problems are sparse {QCQP}s},
   JOURNAL = {Sci. Adv.},
  FJOURNAL = {Science Advances},
    VOLUME = {11},
      YEAR = {2025},
    NUMBER = {1},
     PAGES = {eadl3237},
      hISSN = {2375-2548},
       hDOI = {10.1126/sciadv.adl3237},
       hURL = {https://doi.org/10.1126/sciadv.adl3237},
}

@article {Angeris2019,
    AUTHOR = {Angeris, Guillermo and Vu\v{c}kovi\'{c}, Jelena and Boyd, Stephen P.},
     TITLE = {Computational bounds for photonic design},
   JOURNAL = {ACS Photonics},
  FJOURNAL = {ACS Photonics},
    VOLUME = {6},
      YEAR = {2019},
    NUMBER = {5},
     PAGES = {1232--1239},
      hISSN = {2330-4022},
       hDOI = {10.1021/acsphotonics.9b00154},
       hURL = {https://doi.org/10.1021/acsphotonics.9b00154},
}

@article {Angeris2021,
    AUTHOR = {Angeris, Guillermo and Vu\v{c}kovi\'{c}, Jelena and Boyd, Stephen P.},
     TITLE = {Heuristic methods and performance bounds for photonic design},
   JOURNAL = {Opt. Express},
  FJOURNAL = {Optics Express},
    VOLUME = {29},
      YEAR = {2021},
    NUMBER = {2},
     PAGES = {2827--2854},
      hISSN = {1094-4087},
       hDOI = {10.1364/OE.415052},
       hURL = {https://doi.org/10.1364/OE.415052},
}

@article {Tyburec2019,
    AUTHOR = {Tyburec, M. and Zeman, J. and Nov\'{a}k, J. and Lep\v{s}, M. and Plach\'{y}, T. and Poul, R.},
     TITLE = {Designing modular 3{D} printed reinforcement of wound composite hollow beams with semidefinite programming},
   JOURNAL = {Mater. Des.},
  FJOURNAL = {Materials \& Design},
    VOLUME = {183},
      YEAR = {2019},
     PAGES = {108131},
      hISSN = {0264-1275},
       hDOI = {10.1016/j.matdes.2019.108131},
       hURL = {https://doi.org/10.1016/j.matdes.2019.108131},
}

@article {Dalklint2025,
    AUTHOR = {Dalklint, Anna and Christiansen, Rasmus E. and Sigmund, Ole},
     TITLE = {On the potential use of performance bounds based on semidefinite programs for topology optimization},
   JOURNAL = {Struct. Multidiscip. Optim.},
  FJOURNAL = {Structural and Multidisciplinary Optimization},
    VOLUME = {68},
      YEAR = {2025},
    NUMBER = {7},
     PAGES = {144},
      hISSN = {1615-147X,1615-1488},
       hDOI = {10.1007/s00158-025-04073-0},
       hURL = {https://doi.org/10.1007/s00158-025-04073-0},
}

@book {GiraultRaviart1979,
    AUTHOR = {Girault, V. and Raviart, P.-A.},
     TITLE = {Finite element approximation of the {N}avier--{S}tokes equations},
    SERIES = {Lecture Notes in Mathematics},
    VOLUME = {749},
 PUBLISHER = {Springer-Verlag, Berlin--New York},
      YEAR = {1979},
     hISBN = {3-540-09557-8},
       hDOI = {10.1007/BFb0063447},
       hURL = {https://doi.org/10.1007/BFb0063447},
}

@incollection {Coifman1990,
    AUTHOR = {Coifman, R. and Lions, P.-L. and Meyer, Y. and Semmes, S.},
     TITLE = {Compacit\'{e} par compensation et espaces de {H}ardy},
 BOOKTITLE = {S\'{e}minaire sur les \'{E}quations aux {D}\'{e}riv\'{e}es {P}artielles, 1989--1990},
 PUBLISHER = {\'{E}cole Polytech., Palaiseau},
      YEAR = {1990},
      NOTE = {Exp. No. XIV},
     hISBN = {2-73-020211-0},
      hURL = {https://www.numdam.org/item/SEDP_1989-1990____A16_0/},
}

@article {Clarabel2024,
    AUTHOR = {Goulart, Paul J. and Chen, Yuwen},
     TITLE = {Clarabel: an interior-point solver for conic programs with quadratic objectives},
   JOURNAL = {Math. Program. Comput.},
  FJOURNAL = {Mathematical Programming Computation},
      YEAR = {2026},
      hISSN = {1867-2949,1867-2957},
       hDOI = {10.1007/s12532-026-00320-7},
       hURL = {https://doi.org/10.1007/s12532-026-00320-7},
}

@inproceedings {SumOfSquares2,
    AUTHOR = {Weisser, Tillmann and Legat, Beno\^{\i}t and Coey, Chris and Kapelevich, Lea and Vielma, Juan Pablo},
     TITLE = {Polynomial and moment optimization in {J}ulia and {J}u{MP}},
 BOOKTITLE = {JuliaCon},
      YEAR = {2019},
      hURL = {https://pretalx.com/juliacon2019/talk/QZBKAU/},
}

@inproceedings {SumOfSquares1,
    AUTHOR = {Legat, Beno\^{\i}t and Coey, Chris and Deits, Robin and Huchette, Joey and Perry, Amelia},
     TITLE = {Sum-of-squares optimization in {J}ulia},
 BOOKTITLE = {The First Annual JuMP-dev Workshop},
      YEAR = {2017},
}

@manual {mosek,
    AUTHOR = {{MOSEK ApS}},
     TITLE = {The {MOSEK} Optimizer {API} for {J}ulia. Version 11.1.7},
      YEAR = {2026},
      hURL = {https://docs.mosek.com/11.1/juliaapi/index.html},
}

@article {Chernyavsky2023,
    AUTHOR = {Chernyavsky, Alexander and Bramburger, Jason J. and Fantuzzi, Giovanni and Goluskin, David},
     TITLE = {Convex relaxations of integral variational problems: pointwise dual relaxation and sum-of-squares optimization},
   JOURNAL = {SIAM J. Optim.},
  FJOURNAL = {SIAM Journal on Optimization},
    VOLUME = {33},
      YEAR = {2023},
    NUMBER = {2},
     PAGES = {481--512},
      hISSN = {1052-6234,1095-7189},
       hDOI = {10.1137/21M1455127},
       hURL = {https://doi.org/10.1137/21M1455127},
}

@article {Valmorbida2016,
    AUTHOR = {Valmorbida, Giorgio and Ahmadi, Mohamadreza and Papachristodoulou, Antonis},
     TITLE = {Stability analysis for a class of partial differential equations via semidefinite programming},
   JOURNAL = {IEEE Trans. Automat. Control},
  FJOURNAL = {Institute of Electrical and Electronics Engineers. Transactions on Automatic Control},
    VOLUME = {61},
      YEAR = {2016},
    NUMBER = {6},
     PAGES = {1649--1654},
      hISSN = {0018-9286,1558-2523},
       hDOI = {10.1109/TAC.2015.2479135},
       hURL = {https://doi.org/10.1109/TAC.2015.2479135},
}

@incollection {Korda2022,
    AUTHOR = {Korda, Milan and Henrion, Didier and Lasserre, Jean Bernard},
     TITLE = {Moments and convex optimization for analysis and control of nonlinear {PDE}s},
 BOOKTITLE = {Numerical control. {P}art {A}},
    SERIES = {Handb. Numer. Anal.},
    VOLUME = {23},
     PAGES = {339--366},
 PUBLISHER = {North-Holland, Amsterdam},
      YEAR = {2022},
     hISBN = {978-0-323-85059-9},
       hDOI = {10.1016/bs.hna.2021.12.010},
       hURL = {https://doi.org/10.1016/bs.hna.2021.12.010},
}

@book {Brezis2011,
    AUTHOR = {Brezis, Haim},
     TITLE = {Functional analysis, {S}obolev spaces and partial differential equations},
    SERIES = {Universitext},
 PUBLISHER = {Springer, New York},
      YEAR = {2011},
     hISBN = {978-0-387-70913-0},
       hDOI = {10.1007/978-0-387-70914-7},
       hURL = {https://doi.org/10.1007/978-0-387-70914-7},
}

@article {Lasserre2024,
    AUTHOR = {Lasserre, Jean B.},
     TITLE = {The moment-{SOS} hierarchy: applications and related topics},
   JOURNAL = {Acta Numer.},
  FJOURNAL = {Acta Numerica},
    VOLUME = {33},
      YEAR = {2024},
     PAGES = {841--908},
      hISSN = {0962-4929,1474-0508},
       hDOI = {10.1017/S0962492923000053},
       hURL = {https://doi.org/10.1017/S0962492923000053},
}

@book {Lasserre2010,
    AUTHOR = {Lasserre, Jean Bernard},
     TITLE = {Moments, positive polynomials and their applications},
    SERIES = {Imperial College Press Optimization Series},
    VOLUME = {1},
 PUBLISHER = {Imperial College Press, London},
      YEAR = {2010},
     hISBN = {978-1-84816-445-1; 1-84816-445-9},
       hDOI = {10.1142/p665},
       hURL = {https://doi.org/10.1142/p665},
}

@incollection {Laurent2009,
    AUTHOR = {Laurent, Monique},
     TITLE = {Sums of squares, moment matrices and optimization over polynomials},
 BOOKTITLE = {Emerging applications of algebraic geometry},
    SERIES = {IMA Vol. Math. Appl.},
    VOLUME = {149},
     PAGES = {157--270},
 PUBLISHER = {Springer, New York},
      YEAR = {2009},
     hISBN = {978-0-387-09685-8},
       hDOI = {10.1007/978-0-387-09686-5_7},
       hURL = {https://doi.org/10.1007/978-0-387-09686-5_7},
}

@article {Bellavia2021,
    AUTHOR = {Bellavia, Stefania and Gondzio, Jacek and Porcelli, Margherita},
     TITLE = {A relaxed interior point method for low-rank semidefinite programming problems with applications to matrix completion},
   JOURNAL = {J. Sci. Comput.},
  FJOURNAL = {Journal of Scientific Computing},
    VOLUME = {89},
      YEAR = {2021},
    NUMBER = {2},
     PAGES = {46},
      hISSN = {0885-7474,1573-7691},
       hDOI = {10.1007/s10915-021-01654-1},
       hURL = {https://doi.org/10.1007/s10915-021-01654-1},
}

@article {Habibi2023,
    AUTHOR = {Habibi, Soodeh and Ko\v{c}vara, Michal and Stingl, Michael},
     TITLE = {Loraine -- an interior-point solver for low-rank semidefinite programming},
   JOURNAL = {Optim. Methods Softw.},
  FJOURNAL = {Optimization Methods and Software},
    VOLUME = {39},
      YEAR = {2024},
    NUMBER = {6},
     PAGES = {1185--1215},
      hISSN = {1055-6788,1029-4937},
       hDOI = {10.1080/10556788.2023.2250522},
       hURL = {https://doi.org/10.1080/10556788.2023.2250522},
}

@article {Souto2022,
    AUTHOR = {Souto, Mario and Garcia, Joaquim D. and Veiga, \'{A}lvaro},
     TITLE = {Exploiting low-rank structure in semidefinite programming by approximate operator splitting},
   JOURNAL = {Optimization},
  FJOURNAL = {Optimization. A Journal of Mathematical Programming and Operations Research},
    VOLUME = {71},
      YEAR = {2022},
    NUMBER = {1},
     PAGES = {117--144},
      hISSN = {0233-1934,1029-4945},
       hDOI = {10.1080/02331934.2020.1823387},
       hURL = {https://doi.org/10.1080/02331934.2020.1823387},
}

@article {Burer2006,
    AUTHOR = {Burer, Samuel and Choi, Changhui},
     TITLE = {Computational enhancements in low-rank semidefinite programming},
   JOURNAL = {Optim. Methods Softw.},
  FJOURNAL = {Optimization Methods and Software},
    VOLUME = {21},
      YEAR = {2006},
    NUMBER = {3},
     PAGES = {493--512},
      hISSN = {1055-6788,1029-4937},
       hDOI = {10.1080/10556780500286582},
       hURL = {https://doi.org/10.1080/10556780500286582},
}

@article {Burer2003,
    AUTHOR = {Burer, Samuel and Monteiro, Renato D. C.},
     TITLE = {A nonlinear programming algorithm for solving semidefinite programs via low-rank factorization},
   JOURNAL = {Math. Program.},
  FJOURNAL = {Mathematical Programming. A Publication of the Mathematical Programming Society},
    VOLUME = {95},
      YEAR = {2003},
    NUMBER = {2},
     PAGES = {329--357},
      hISSN = {0025-5610,1436-4646},
       hDOI = {10.1007/s10107-002-0352-8},
       hURL = {https://doi.org/10.1007/s10107-002-0352-8},
}

@article {Monteiro2026,
    AUTHOR = {Monteiro, Renato D. C. and Sujanani, Arnesh and Cifuentes, Diego},
     TITLE = {A low-rank augmented {L}agrangian method for large-scale semidefinite programming based on a hybrid convex-nonconvex approach},
   JOURNAL = {Math. Program.},
  FJOURNAL = {Mathematical Programming. A Publication of the Mathematical Programming Society},
      YEAR = {2026},
      hISSN = {0025-5610,1436-4646},
       hDOI = {10.1007/s10107-026-02338-0},
       hURL = {https://doi.org/10.1007/s10107-026-02338-0},
}

%%%%%%%%%%%%%%%%%%%%%%%%%%%%%%%%%%%%%%%%%%%%%%%%%%%%%%%%%%%%%%%%%%%%%%%%%%%%%%%%
%%%%%%%%%%%%%%%%%%%%%%%%%%%%%%%%%%%%%%%%%%%%%%%%%%%%%%%%%%%%%%%%%%%%%%%%%%%%%%%%
\end{document}